\pdfoutput=1
\documentclass[]{interact}

\usepackage{mathtools}
\usepackage{array}
\usepackage{tabularx}
\usepackage{multirow}
\usepackage{algorithm,algpseudocode}
\usepackage[numbers,sort&compress]{natbib}
\bibpunct[, ]{[}{]}{,}{n}{,}{,}
\renewcommand\bibfont{\fontsize{10}{12}\selectfont}
\usepackage[colorlinks=true,linkcolor=blue,citecolor=blue,urlcolor=blue]{hyperref}
\usepackage[nameinlink,noabbrev]{cleveref}

\DeclareMathOperator{\prox}{prox}

\theoremstyle{plain}
\newtheorem{theorem}{Theorem}[section]
\newtheorem{lemma}[theorem]{Lemma}
\newtheorem{proposition}[theorem]{Proposition}
\newtheorem{corollary}[theorem]{Corollary}

\theoremstyle{definition}
\newtheorem{definition}[theorem]{Definition}
\newtheorem{assumption}[theorem]{Assumption}

\theoremstyle{remark}
\newtheorem{remark}[theorem]{Remark}

\begin{document}

\title{A Decomposed Bilevel Search for Variable-Metric Proximal Gradient Methods}

\author{
\name{Xinpeng Li\textsuperscript{a}\thanks{CONTACT Xinpeng Li. Email: \texttt{lixinpeng23@mails.ucas.ac.cn}} and Ya-xiang Yuan\textsuperscript{b}}
\affil{\textsuperscript{a}State Key Laboratory of Mathematical Sciences, Academy of Mathematics and Systems Science, Chinese Academy of Sciences, and University of Chinese Academy of Sciences, China; \textsuperscript{b}State Key Laboratory of Mathematical Sciences, Academy of Mathematics and Systems Science, Chinese Academy of Sciences, China}
}

\maketitle

\begin{abstract}
Variable-metric proximal methods accelerate composite convex optimization,
  but the scaled proximal map induced by a quasi-Newton metric rarely
  has a closed form.  We develop \emph{Decomposed Bilevel Search} (DBS), based
  on a diagonal-plus-rank-one factor \(X=D+uv^\top\) whose diagonal scaling
  satisfies a weak secant equation.  The induced inverse metric
  \(B^{-1}=XX^\top\) recovers zero-memory DFP/BFGS-type Broyden members, while
  the factor form reduces each scaled proximal step to a two-dimensional
  monotone residual system.  Each residual evaluation requires one
  diagonal-metric proximal map, and a certified bilevel solve reaches target
  accuracy \(\epsilon\) in
  \(\mathcal O((d+T_p)\log^2(1/\epsilon))\) work, where \(T_p\) is the cost of
  that proximal map.  Under strong convexity, the outer method converges
  linearly with exact and inexact inner solves.  In the scalar
  specialization \(D=\alpha I\), the oracle uses only ordinary proximal
  evaluations of the regularizer and no generalized Jacobian or active-set
  information.

  Experiments on ordered-weighted \(\ell_1\) (SLOPE/OWL) and group-lasso
  logistic regression evaluate both the inner oracle and the full outer
  method.  The oracle solves
  high-dimensional SLOPE scaled proximal subproblems of
  condition number up to \(4\times10^6\) using a few hundred ordinary
  proximal evaluations.  In SLOPE outer benchmarks, the
  warm-started oracle keeps the scaled-proximal overhead controlled and DBS
  reaches stringent targets with far fewer gradient evaluations than
  Lipschitz-normalized FISTA; on
  \texttt{real-sim} this becomes a clear target-time advantage.  On
  group-lasso logistic regression, DBS is reliable on synthetic correlated
  instances and fastest on a real-data grouped-copy instance.
\end{abstract}


\begin{amscode}
90C25, 90C30, 90C53, 65K05
\end{amscode}

\section{Introduction}
  We consider the composite convex minimization problem
  \begin{equation}\label{eq:1.1}
      \min_{x\in\mathbb{R}^d} F(x):=f(x)+h(x),
  \end{equation}
  where \(f:\mathbb{R}^d\to\mathbb{R}\) is convex and continuously
  differentiable with \(L\)-Lipschitz continuous gradient, and
  \(h\in\Gamma_0(\mathbb{R}^d)\) is possibly nonsmooth. Here
  \(\Gamma_0(\mathbb{R}^d)\) denotes the class of proper, lower semicontinuous
  convex functions from \(\mathbb{R}^d\) to
  \(\mathbb{R}\cup\{+\infty\}\). We assume that the solution set
  \(\arg\min F\) is nonempty, and let \(x^*\in\arg\min F\).

  Problem~\eqref{eq:1.1} covers many regularized learning and signal-processing
  models, including elastic-net regression, sparse classification with logistic
  or squared-hinge losses, and structured penalties such as sorted
  \(\ell_1\), group, total-variation, and spectral regularizers.  In many of
  these models the ordinary proximal mapping of \(h\) is inexpensive, while proximal mappings in non-Euclidean quasi-Newton metrics are
  substantially difficult.

  A simple and scalable approach for~\eqref{eq:1.1} is the proximal gradient
  method~\cite{combettes2011proximal},
  \[
      x_{k+1}
      =
      \operatorname{prox}_{\eta_k h}
      \bigl(x_k-\eta_k\nabla f(x_k)\bigr),
  \]
  where
  \[
      \operatorname{prox}_{\eta h}(x)
      =
      \arg\min_y
      \left\{
          h(y)+\frac{1}{2\eta}\|y-x\|_2^2
      \right\}.
  \]
  Its accelerated variant FISTA~\cite{beck2009fast} achieves the optimal
  \(\mathcal O(1/k^2)\) convergence rate for convex objectives, and linear
  convergence is available under strong convexity and standard stepsize
  conditions \cite{nesterov2013gradient}.

  Despite their favorable complexity guarantees, first-order proximal methods
  often converge slowly on ill-conditioned problems. This motivates the use of
  curvature information in the proximal step. A standard approach replaces the
  Euclidean metric by a positive definite matrix \(B_k\), leading to the
  variable-metric iteration
  \begin{equation}\label{eq:1.4}
      x_{k+1}
      =
      \arg\min_{x\in\mathbb{R}^d}
      \left\{
          \langle g_k,x-x_k\rangle
          +
          \frac{1}{2\eta_k}
          \langle x-x_k,B_k(x-x_k)\rangle
          +
          h(x)
      \right\},
  \end{equation}
  where \(g_k=\nabla f(x_k)\) and \(B_k\in\mathbb{S}_{++}^{d\times d}\).
  General variable-metric forward--backward methods study such iterations under
  bounded metrics, inexact proximal steps, line searches, and extrapolation
  while preserving convergence guarantees
  \cite{bonettini2016variable,bonettini2016extrapolation,salzo2017variable}.
  When \(B_k\) is chosen as the Hessian or as a suitable quasi-Newton
  approximation satisfying standard Dennis--Mor\'e type conditions, proximal
  Newton and proximal quasi-Newton methods can achieve fast local convergence
  \cite{lee2014proximal}. A natural quasi-Newton requirement is the secant
  condition
  \begin{equation}\label{eq:1.8}
      B_k s_{k-1}=y_{k-1},
      \qquad
      s_{k-1}=x_k-x_{k-1},\quad
      y_{k-1}=g_k-g_{k-1}.
  \end{equation}
  
  The main computational difficulty lies in the evaluation of the scaled proximal
  mapping
  \begin{equation}\label{eq:1.7}
      x_{k+1}
      =
      \operatorname{prox}_{\eta_k h}^{B_k}
      \bigl(x_k-\eta_k B_k^{-1}g_k\bigr),
  \end{equation}
  which generally has no closed-form expression for nontrivial metrics \(B_k\).
  Derivative-based
  inner solvers, such as semismooth Newton or spectral projected-gradient
  methods, can be very effective when the generalized derivative structure of
  the proximal operator is available, but this structure is often
  regularizer-specific \cite{andrew2007scalable,becker2011templates}.

  A common way to avoid this inner-solver cost is to restrict the metric so that
  the scaled proximal map has an explicit or low-cost form.  Barzilai--Borwein
  and related spectral-gradient rules provide especially inexpensive scalar
  scalings
  \cite{barzilai1988two,raydan1997barzilai,birgin2000nonmonotone}.  Recent
  variable-metric proximal-gradient variants also use inexpensive diagonal or
  coordinatewise scalings
  \cite{park2020variable,yu2023mini}.  These scalar and diagonal
  proximal-gradient methods are usually inexpensive and empirically effective.
  However, a purely scalar or diagonal metric encodes only spectral or weak-secant curvature information. It cannot represent the curvature corrections produced by quasi-Newton updates.

  Becker and Fadili introduced a zero-memory symmetric rank-one method in which
  a diagonal plus/minus rank-one metric reduces the scaled proximal mapping to
  a one-dimensional equation for separable regularizers
  \cite{becker2012quasi}. They subsequently extended this
   to diagonal plus/minus rank-\(r\) metrics
  \cite{becker2019quasi}.  These works show that low-rank metric structure can
  make variable-metric proximal steps tractable.  Our motivation is to retain a
  proximal-oracle interface for structured regularizers whose proximal maps are
  available but whose generalized Jacobians or active-set descriptions are
  inconvenient.  Instead of inserting the low-rank metric directly into the
  proximal KKT system, we work with a diagonal-plus-rank-one factor
  \(X=D+uv^\top\).  This factor-coordinate viewpoint leads to a
  two-dimensional monotone residual system whose evaluation requires one
  diagonal-metric proximal map and no generalized Jacobian information from the
  proximal operator.

  The main contributions of this paper are summarized as follows.
  \begin{itemize}
     \item We propose diagonal-plus-rank-one factorizations for diagonally
      initialized zero-memory quasi-Newton metrics.  The resulting
      inverse-side and Hessian-side metrics satisfy secant equations for the
      damped curvature pair and recover zero-memory DFP/BFGS-type Broyden
      members.

      \item We reduce the scaled proximal subproblem induced by these factors
      to a two-dimensional monotone system whose residual evaluation requires
      exactly one diagonal-metric proximal evaluation.

      \item We give a certified bilevel solve and a safeguarded oracle for this
      system, with worst-case cost
      \(\mathcal O((d+T_p)\log(R_t/\epsilon)\log(R_{\xi}/\epsilon))\), or
      \(\mathcal O((d+T_p)\log^2(1/\epsilon))\) for fixed problem data.  The
      outer oracle includes a warm start and a relative localization tolerance
      while retaining the certified bilevel fallback.

      \item We give admissibility conditions that ensure uniform positive
      definiteness, linear convergence under strong convexity, and the
      corresponding overall complexity bound for two-sided DBS methods.

      \item We conduct numerical experiments on ordered-weighted \(\ell_1\)
      (SLOPE/OWL) and group-lasso logistic regression.  The results isolate
      the scaled proximal oracle, quantify the warm-started outer cost on
      SLOPE, and show that the iteration advantage can become a target-time
      advantage on real-data SLOPE and group-lasso instances.
  \end{itemize}

The rest of the paper is organized as follows.
  Section~\ref{sec:prelim} gives notation and preliminary proximal facts.
  Section~\ref{sec:factorization-search} develops the DBS proximal oracle and
  the associated Broyden-type factorizations.
  Section~\ref{sec:two-sided-dbs} presents the two-sided DBS algorithmic
  framework and the convergence and complexity analysis.
  Section~\ref{sec:experiments} reports the numerical experiments, and
  Section~\ref{sec:conclusions} concludes the paper.

\section{Notation and Preliminaries}
  \label{sec:prelim}

  Throughout the paper, \(\|\cdot\|\) denotes the Euclidean norm for vectors and
  the induced spectral norm for matrices, unless otherwise specified.  For a
  symmetric positive definite matrix \(B\), we use
  \[
      \langle x,y\rangle_B=x^\top By,
      \qquad
      \|x\|_B^2 = x^\top Bx .
  \]
  The identity matrix is denoted by \(I\), and
  \(\mathbb S_{++}^{d\times d}\) denotes the set of symmetric positive definite
  \(d\times d\) matrices.  For a scalar \(t\), we write
  \([t]_+=\max\{t,0\}\).  For \(P\in\mathbb S_{++}^{d\times d}\), the
  \(P\)-metric proximal mapping with parameter \(\sigma>0\) is
  \[
      \prox_{\sigma h}^{P}(x)
      =
      \arg\min_z
      \left\{
          \frac12\|z-x\|_P^2+\sigma h(z)
      \right\}.
  \]

   We also recall several standard facts about metric proximal mappings and
   elementary convex analysis \cite{bauschkeCombettes2017} that are key to our
   analysis.

  \begin{proposition}[Variational characterization of the proximal mapping]
  \label{prop:prox-variational}
  Let \(h\in\Gamma_0(\mathbb R^d)\), \(\sigma>0\), and
  \(P\in\mathbb S_{++}^{d\times d}\).  For \(x,p\in\mathbb R^d\),
  \[
      p=\prox_{\sigma h}^{P}(x)
  \quad\Longleftrightarrow\quad
      \langle y-p,x-p\rangle_P
      \le
      \sigma\bigl(h(y)-h(p)\bigr),
      \qquad
      \forall y\in\mathbb R^d .
  \]
  \end{proposition}

  \begin{corollary}[Optimality condition]
  \label{cor:prox-opt}
  Let \(h\in\Gamma_0(\mathbb R^d)\), \(\sigma>0\), and
  \(P\in\mathbb S_{++}^{d\times d}\).  For \(x,p\in\mathbb R^d\),
  \[
      p=\prox_{\sigma h}^{P}(x)
      \quad\Longleftrightarrow\quad
      \frac{1}{\sigma}P(x-p)\in\partial h(p).
  \]
  \end{corollary}

  \begin{corollary}[Firm nonexpansiveness]
  \label{cor:prox-firm}
  Let \(h\in\Gamma_0(\mathbb R^d)\), \(\sigma>0\), and
  \(P\in\mathbb S_{++}^{d\times d}\).  For \(x,y\in\mathbb R^d\),
  \[
      \|\prox_{\sigma h}^{P}(x)-\prox_{\sigma h}^{P}(y)\|_P^2
      \le
      \left\langle
      \prox_{\sigma h}^{P}(x)-\prox_{\sigma h}^{P}(y),
      x-y
      \right\rangle_P .
  \]
  \end{corollary}

  \begin{proposition}[Envelope rule]
  \label{prop:envelope}
  Let \(\Phi(q,\xi)\) be proper, closed, and strongly convex in \(q\) for each
  \(\xi\in\mathbb R^m\).  Suppose that \(\Phi(q,\xi)\) is differentiable in
  \(\xi\) for each fixed \(q\), and let
  \[
      q(\xi)=\arg\min_q \Phi(q,\xi),
      \qquad
      \varphi(\xi)=\min_q \Phi(q,\xi).
  \]
  Then we have
  \[
      \nabla \varphi(\xi)
      =
      \nabla_\xi \Phi(q(\xi),\xi).
  \]
  This is the differentiable envelope rule for a value function with a unique
  minimizer.
  \end{proposition}

  \begin{lemma}[Monotonicity of saddle gradients]
  \label{lem:saddle-monotone}
  Let \(\ell(t,\xi)\) be convex in \(t\in\mathbb R\) and concave in
  \(\xi\in\mathbb R\), and suppose that the partial derivatives used below
  exist.  Then
  \[
      \mathcal T(t,\xi)
      =
      \bigl(\partial_t\ell(t,\xi),-\partial_\xi\ell(t,\xi)\bigr)
  \]
  is monotone on \(\mathbb R^2\).
  \end{lemma}
  \begin{proof}
  Let \(g_{t,i}=\partial_t\ell(t_i,\xi_i)\) and
  \(g_{\xi,i}=\partial_\xi\ell(t_i,\xi_i)\).  Convexity in \(t\)
  gives
  \[
      (g_{t,1}-g_{t,2})(t_1-t_2)
      \ge
      \ell(t_1,\xi_1)-\ell(t_2,\xi_1)
      +
      \ell(t_2,\xi_2)-\ell(t_1,\xi_2).
  \]
  Concavity in \(\xi\) gives
  \[
      -(g_{\xi,1}-g_{\xi,2})(\xi_1-\xi_2)
      \ge
      \ell(t_1,\xi_2)-\ell(t_1,\xi_1)
      +
      \ell(t_2,\xi_1)-\ell(t_2,\xi_2).
  \]
  Adding the two inequalities yields
  \[
      \left\langle
      \mathcal T(t_1,\xi_1)-\mathcal T(t_2,\xi_2),
      (t_1-t_2,\xi_1-\xi_2)
      \right\rangle\ge0 .
  \]
  \end{proof}
  
\section{Factorization and Bilevel Search}
\label{sec:factorization-search}

In this section, we develop a technique for computing the metric proximal
mapping \(\prox_{\eta h}^{B}\) (with stepsize \(\eta>0\)).   First, we assume that the inverse metric
admits a nonsingular factorization \(B^{-1}=XX^\top\) with
\begin{equation}\label{eq:factorization}
X = D + u v^\top,\qquad D=\operatorname{diag}(d_i)\succ0,\; u,v\in\mathbb{R}^d.
\end{equation}

Then \(\prox_{\eta h}^{B}\) can be expressed in terms of a standard proximal operator composed with \(X\):
\[
\prox_{\eta h}^{B} = X \circ \prox_{\eta h\circ X} \circ X^{-1},
\]
where \((h\circ X)(z)=h(Xz)\).  Indeed, let \(p = X^{-1}x\) and change variable \(y = X z\), we have
\[
\begin{aligned}
\prox_{\eta h}^{B}(x)
&= \arg\min_{y}\Bigl\{\frac12\|y-x\|_B^2 + \eta h(y)\Bigr\} \\
&= X\circ\arg\min_{z}\Bigl\{\frac12\|z-p\|^2 + \eta h(Xz)\Bigr\} \\
&= X\circ\prox_{\eta h\circ X}(X^{-1}x).
\end{aligned}
\]
Thus the computation of a metric proximal step reduces to that of
\(\prox_{\eta h\circ X}( \cdot )\), where \(X\) is a rank-one perturbation of a
diagonal matrix.

\subsection{Two-dimensional monotone search}
\label{sub:monotone-search}
In this section, we describe how to compute \(\prox_{\eta h\circ X}\) for
\(X = D + u v^\top\).  If \(v=0\), then \(X=D\), and the computation reduces
to a diagonal-metric proximal evaluation.  In the following computation we
assume \(v\neq0\).  Define
\[
\Psi_x(z) = \frac12\|z-x\|^2 + \eta h(Dz + u\,v^\top z).
\]
Because \(X\) is nonsingular and \(h\) is proper, \(\Psi_x\) has nonempty
domain.  The quadratic term makes \(\Psi_x\) strongly convex and coercive, so
\(\prox_{\eta h\circ X}(x)=\arg\min_z\Psi_x(z)\) exists and is unique.

The rank-one factorization introduces two scalar variables.  Let
\[
    t=v^\top z,\qquad q=Dz+ut,
    \qquad c=1+\langle D^{-1}v,u\rangle .
\]
Since \(X\) is nonsingular, \(c\neq0\).  Conversely,
\[
    z=D^{-1}(q-ut),
\]
and the identity \(t=v^\top z\) is equivalent to the single linear constraint
\[
    \langle D^{-1}v,q\rangle=ct .
\]
Hence the scaled proximal subproblem is equivalent to
\begin{equation}\label{eq:qt-subproblem}
    \min_{q,t}\;
    \frac12\|D^{-1}(q-ut)-x\|^2+\eta h(q)
    \quad
    \text{subject to}\quad
    \langle D^{-1}v,q\rangle=ct .
\end{equation}

We now derive the optimality system for~\eqref{eq:qt-subproblem}.  Since
\(c\neq0\), for every \(q\) there exists a scalar \(t\) such that
\(\langle D^{-1}v,q\rangle=ct\).  Hence
\[
    \{q:\exists t,\ \langle D^{-1}v,q\rangle=ct\}
    \cap \operatorname{ri}(\operatorname{dom}h)\neq\emptyset
\]
Therefore the standard affine-constraint qualification holds automatically in
these factorized coordinates, and the KKT conditions below are necessary and
sufficient. Introduce a Lagrange multiplier \(\xi\in\mathbb R\) and define
\[
\mathcal L_x(q,t,\xi)
=
\frac12\|D^{-1}(q-ut)-x\|^2+\eta h(q)
+\xi\bigl(\langle D^{-1}v,q\rangle-ct\bigr).
\]
For fixed \((t,\xi)\), the \(q\)-subproblem is strongly convex.  Its
first-order condition is
\[
0\in
D^{-1}\bigl(D^{-1}(q-ut)-x\bigr)
+\xi D^{-1}v+\eta\partial h(q),
\]
or, equivalently,
\begin{equation}\label{eq:q-stationarity}
    D^{-2}\bigl(Dx+tu-\xi Dv-q\bigr)
    \in \eta\partial h(q).
\end{equation}
Let \(a(t,\xi)=Dx+tu-\xi Dv\), by the metric proximal optimality condition in
Corollary~\ref{cor:prox-opt}, with \(P=D^{-2}\), condition
\eqref{eq:q-stationarity} is equivalent to
\begin{equation}\label{eq:qtxi-oracle}
   q(t,\xi)=\prox_{\eta h}^{D^{-2}}\bigl(a(t,\xi)\bigr),
\end{equation}
so \(q(t,\xi)\) is the unique minimizer of
\(\mathcal L_x(\cdot,t,\xi)\).  We then set
\(z(t,\xi)=D^{-1}\bigl(q(t,\xi)-ut\bigr)\).
Since
\[
    \partial_t
    \frac12\|D^{-1}(q-ut)-x\|^2
    =
    \langle D^{-1}u,x-D^{-1}(q-ut)\rangle,
\]
the remaining KKT equations can be written as the residual system
\begin{equation}\label{eq:two-dim-residual}
\begin{aligned}
    r_1(t,\xi)
    &=
    \langle D^{-1}v,q(t,\xi)\rangle-ct,\\
    r_2(t,\xi)
    &=
    \langle D^{-1}u,x-z(t,\xi)\rangle-c\xi .
\end{aligned}
\end{equation}
The first residual enforces the rank-one consistency constraint, and the second
residual is the stationarity condition in the scalar variable \(t\).
Moreover, zeros of the residual system are exactly the scaled proximal points:
\[
    r_1(t^*,\xi^*)=r_2(t^*,\xi^*)=0
    \quad\Longleftrightarrow\quad
    z(t^*,\xi^*)=\prox_{\eta h\circ X}(x).
\]
Indeed, \eqref{eq:qtxi-oracle} gives \(q\)-stationarity, while \(r_1=0\) and
\(r_2=0\) give primal feasibility and \(t\)-stationarity.  By the
constraint qualification above, these are precisely the KKT conditions for
\eqref{eq:qt-subproblem}.  Conversely, if
\(z^*=\prox_{\eta h\circ X}(x)\), then
\[
    t^*=v^\top z^*,
    \qquad
    q^*=Xz^*,
\]
and the KKT theorem gives a multiplier \(\xi^*\), so
\((t^*,\xi^*)\) is a zero of~\eqref{eq:two-dim-residual}.

\medskip \noindent\textbf{Monotonicity.}
The residual system has the structure needed for monotone localization: it is
two-dimensional, monotone, and its oracle evaluation requires one
diagonal-metric proximal map.  Define
\[
    \mathcal T(t,\xi)=\bigl(r_2(t,\xi),-r_1(t,\xi)\bigr).
\]
Let
\[
\ell_x(t,\xi)
=
\min_q
\mathcal L_x(q,t,\xi).
\]
The function \(\ell_x\) is convex in \(t\) and concave in \(\xi\): for fixed
\(\xi\), convexity follows by partial minimization of a jointly convex
function in \((q,t)\), while for fixed \(t\), concavity follows because
\(\ell_x\) is the pointwise infimum of affine functions of \(\xi\).  For
each fixed \((t,\xi)\), the dependence of \(\mathcal L_x\) on \(q\)
contains the strongly convex quadratic term
\(\frac12\|D^{-1}q-D^{-1}ut-x\|^2\), whose Hessian is \(D^{-2}\succ0\).  Hence
the minimizing \(q\) is unique and is precisely \(q(t,\xi)\) from
\eqref{eq:qtxi-oracle}.  The envelope rule in
Proposition~\ref{prop:envelope} gives
\[
    \partial_t\ell_x(t,\xi)=r_2(t,\xi),
    \qquad
    \partial_\xi\ell_x(t,\xi)=r_1(t,\xi).
\]
Therefore
\[
    \mathcal T
    =
    (\partial_t\ell_x,-\partial_\xi\ell_x),
\]
and Lemma~\ref{lem:saddle-monotone} implies that \(\mathcal T\) is monotone:
\begin{equation}\label{eq:monotone-T}
    \left\langle
    \mathcal T(t_1,\xi_1)-\mathcal T(t_2,\xi_2),
    (t_1-t_2,\xi_1-\xi_2)
    \right\rangle\ge0 .
\end{equation}
If \((t^*,\xi^*)\) is a zero of \(\mathcal T\), then every trial point
\((t,\xi)\) generates a valid separating halfspace,
\[
    \left\langle
    \mathcal T(t,\xi),
    (t^*,\xi^*)-(t,\xi)
    \right\rangle\le0 .
\]
Thus a two-dimensional cutting-plane localization routine can use these cuts to
search for a zero of \(\mathcal T\) using only evaluations of the residual
oracle.

\begin{algorithm}[t]
\caption{Two-dimensional monotone search for
\(\prox_{\eta h\circ(D+uv^\top)}(x)\)}
\label{alg:monotone-prox}
\begin{algorithmic}[1]
\Require \(x,\eta,D,u,v\), an initial rectangle
\(\mathcal B_0\) from Lemma~\ref{lem:search-rectangle}, a localization
tolerance \(\delta_{\rm loc}>0\), and a residual tolerance
\(\epsilon_{\rm res}>0\).
\State Initialize an ellipse \(\mathcal E_0\supseteq\mathcal B_0\).
\For{\(j=0,1,\ldots\)}
    \State Let \((t_j,\xi_j)\) be the center of \(\mathcal E_j\).
    \State Compute
    \(q_j=\prox_{\eta h}^{D^{-2}}(Dx+t_j u-\xi_j Dv)\) and
    \(z_j=D^{-1}(q_j-u t_j)\).
    \State Compute \(r_{1,j},r_{2,j}\) from~\eqref{eq:two-dim-residual} and set
    \(\mathcal T_j=(r_{2,j},-r_{1,j})\).
    \If{\(\|(r_{1,j},r_{2,j})\|\le\epsilon_{\rm res}\) or
    \(\operatorname{diam}(\mathcal E_j)\le\delta_{\rm loc}\)}
        \State \Return \(z_j\).
    \EndIf
    \State Set
    \[
        \mathcal H_j
        =
        \left\{
        (t,\xi):
        \left\langle
        \mathcal T_j,(t,\xi)-(t_j,\xi_j)
        \right\rangle\le0
        \right\}.
    \]
    \State Update \(\mathcal E_{j+1}\supseteq\mathcal E_j\cap\mathcal H_j\)
    using a two-dimensional localization routine (e.g., an enclosing-ellipse
    cutting-plane update).
\EndFor
\end{algorithmic}
\end{algorithm}

For reproducibility, a standard central-cut ellipsoid update can be used for the
enclosing-ellipse step.  If
\(\mathcal E_j=\{w:(w-c_j)^\top A_j^{-1}(w-c_j)\le1\}\) in \(\mathbb R^2\) and
the cut is \(a_j^\top(w-c_j)\le0\), set
\[
    p_j=\frac{A_j a_j}{\sqrt{a_j^\top A_j a_j}},\qquad
    c_{j+1}=c_j-\frac13p_j,\qquad
    A_{j+1}=\frac43\left(A_j-\frac23p_jp_j^\top\right).
\]
Then \(\mathcal E_{j+1}\) contains \(\mathcal E_j\cap\mathcal H_j\).  Other
valid two-dimensional cutting-plane localization updates can be substituted.

\medskip\noindent\textbf{Basic residual properties.}

\begin{lemma}[Uniqueness of the residual zero]\label{lem:residual-unique}
The residual system~\eqref{eq:two-dim-residual} has a unique zero
\((t^*,\xi^*)\), given by \(t^*=v^\top z^*\) and
\(\xi^*=\langle D^{-1}u,x-z^*\rangle/c\), where
\(z^*=\prox_{\eta h\circ X}(x)\).
\end{lemma}
\begin{proof}
The scaled proximal point \(z^*=\arg\min_z\Psi_x(z)\) is unique because
\(\Psi_x\) is strongly convex.  By the equivalence
of~\eqref{eq:qt-subproblem}--\eqref{eq:two-dim-residual}, every zero
\((t,\xi)\) of~\eqref{eq:two-dim-residual} yields \(z(t,\xi)=z^*\).  Since
\(t=v^\top z(t,\xi)\) at a zero, \(t=v^\top z^*=:t^*\) for every zero.
Substituting \(z(t^*,\xi)=z^*\) into \(r_2(t^*,\xi)=0\) gives
\(\langle D^{-1}u,x-z^*\rangle-c\xi=0\).  Because \(X\) is nonsingular,
\(c=1+\langle D^{-1}v,u\rangle\neq0\), so \(\xi\) is unique.
\end{proof}

\begin{lemma}[Scalar localization controls the proximal point]
\label{lem:scalar-localization}
Let \((t^*,\xi^*)\) be the zero from Lemma~\ref{lem:residual-unique}.  Then
for all \((t,\xi)\),
\[
    \|z(t,\xi)-z(t^*,\xi^*)\|
    \le C_X\|(t,\xi)-(t^*,\xi^*)\|,
    \qquad
    C_X=\sqrt{4\|D^{-1}u\|^2+\|v\|^2}.
\]
\end{lemma}
\begin{proof}
Set \(\Delta t=t-t^*\) and \(\Delta\xi=\xi-\xi^*\).  Since
\(\prox_{\eta h}^{D^{-2}}\) is nonexpansive in the \(D^{-2}\)-metric,
\[
    \|D^{-1}(q(t,\xi)-q(t^*,\xi^*))\|
    \le
    \|D^{-1}u\,\Delta t-v\,\Delta\xi\|.
\]
Consequently, with \(z(t,\xi)=D^{-1}(q(t,\xi)-ut)\),
\begin{equation}\label{eq:z-lipschitz}
\begin{aligned}
    \|z(t,\xi)-z(t^*,\xi^*)\|
    &\le
    2\|D^{-1}u\|\,|\Delta t|
    +\|v\|\,|\Delta\xi|  \\
    &\le C_X\|(t,\xi)-(t^*,\xi^*)\|.
\end{aligned}
\end{equation}

\end{proof}

Thus reducing the scalar accuracy to \(\epsilon_z/C_X\) gives an
\(O(\epsilon_z)\)-accurate scaled proximal point.

\medskip \noindent\textbf{Finite search rectangle.}
The following lemma provides explicit a priori bounds for the two scalar
unknowns in the monotone system.

\begin{lemma}[Explicit search rectangle]\label{lem:search-rectangle}
Assume that \(h\) has a known lower bound \(\underline h\) and
\(h(Xx)<\infty\).  Let
\((t^*,\xi^*)\) be the zero from Lemma~\ref{lem:residual-unique}.  Then
\[
    t^* \in [t_c - R_t,\; t_c + R_t],
    \qquad
    \xi^*\in[-R_\xi,\;R_\xi],
\]
where
\[
    t_c=\langle v,x\rangle,
    \qquad
    R_t=\|v\|\sqrt{2\eta (h(Xx)-\underline h)},
\]
and
\[
    R_\xi
    =
    \frac{\|D^{-1}u\|}{|c|}
    \sqrt{2\eta (h(Xx)-\underline h)} ,
    \qquad
    c=1+\langle D^{-1}v,u\rangle .
\]
\end{lemma}

\begin{proof}
Let \(z^*=z(t^*,\xi^*)\).  Since \(z^*\) minimizes \(\Psi_x\),
\[
    \Psi_x(z^*)\le \Psi_x(x)=\eta h(Xx).
\]
Because \(h(Xz^*)\ge\underline h\),
\[
    \|z^*-x\|\le \sqrt{2\eta (h(Xx)-\underline h)} .
\]
The first bound follows from
\[
    |t^*-\langle v,x\rangle|
    =
    |\langle v,z^*-x\rangle|
    \le
    \|v\|\|z^*-x\|.
\]
The second residual equation gives
\[
    c\xi^*
    =
    \langle D^{-1}u,x-z^*\rangle ,
\]
and therefore
\[
    |\xi^*|
    \le
    \frac{\|D^{-1}u\|}{|c|}
    \|z^*-x\|
    \le
    R_\xi .
\]
\end{proof}

\subsection{Certified bilevel solve and a safeguarded oracle}
\label{sub:bilevel-search}
Algorithm~\ref{alg:monotone-prox} is the practical localization phase of the
DBS oracle.  The certified solve below supplies the worst-case guarantee: given
finite brackets for \(t^*\) and for the inner \(\xi\)-solve, it returns an
\(\epsilon\)-accurate scaled proximal point using
\(\mathcal O(\log(R_t/\epsilon)\log(R_{\xi}/\epsilon))\)
residual evaluations.
The safeguarded oracle runs the localization phase first and calls this
certified solve only if the localization certificate is not obtained within the
prescribed budget.

\medskip\noindent\textbf{Certified bilevel solve.}
The certified solve uses two one-dimensional monotonicities inherited from the
saddle structure.  For fixed \(t\), \(r_1(t,\cdot)\) is nonincreasing, so a
sign bracket \(J_t=[\xi^-_t,\xi^+_t]\) with
\(r_1(t,\xi^-_t)\ge0\ge r_1(t,\xi^+_t)\) can be refined by bisection.  After
this inner elimination, the corresponding \(r_2\)-values form a nondecreasing
outer subgradient in \(t\), so a certified sign of \(r_2\) gives the outer
bisection direction.  Algorithm~\ref{alg:bilevel-certified} certifies this
sign using an interval enclosure
\(\mathcal I_2(t,J_t)\supset\{r_2(t,\xi):\xi\in J_t\}\); if the sign is not
separated and \(J_t\) is already \(O(\epsilon)\), the same enclosure gives an
\(O(\epsilon)\) residual certificate.
The norm \(\|(r_1,r_2)\|\) is also used as a computable reduced KKT residual:
when it is \(O(\epsilon)\), the residual-to-solution conversion in
Lemma~\ref{lem:bilevel-complexity} gives an \(O(\epsilon)\)-accurate scaled
proximal point.
The small-\(J_t\) branch also controls the first residual: by nonexpansiveness
of the diagonal-metric proximal map, for fixed \(t\),
\begin{equation}\label{eq:r1-xi-lipschitz}
\begin{aligned}
    |r_1(t,\xi)-r_1(t,\xi')|
    &\le
    \|v\|\,
    \|D^{-1}(q(t,\xi)-q(t,\xi'))\|  \\
    &\le \|v\|^2|\xi-\xi'|.
\end{aligned}
\end{equation}
Thus an \(O(\epsilon)\)-wide inner bracket gives
\(r_1(t,\hat\xi)=O(\epsilon)\) at its midpoint.  The second residual is
controlled separately by the interval enclosure \(\mathcal I_2(t,J_t)\).

\begin{algorithm}[t]
\caption{Certified bilevel solve for the rank-one scaled proximal step}
\label{alg:bilevel-certified}
\begin{algorithmic}[1]
\Require Outer bracket \([t^-,t^+]\) containing \(t^*\); inner center
\(\bar\xi(t)\) (e.g., \(0\)); initial inner radius \(\rho_0>0\); tolerance
\(\epsilon\).
\While{\(t^+-t^- > \epsilon\)}
    \State Set \(t=(t^-+t^+)/2\), \(\rho=\rho_0\), and
    \(J_t=[\bar\xi(t)-\rho,\bar\xi(t)+\rho]\).
    \While{\(r_1(t,\xi^-_t)<0\) or \(r_1(t,\xi^+_t)>0\)}
        \State Set \(\rho=2\rho\) and
        \(J_t=[\bar\xi(t)-\rho,\bar\xi(t)+\rho]\).
    \EndWhile
    \Loop
        \State Set \(\hat\xi=\operatorname{mid}(J_t)\), compute
        \(r_1(t,\hat\xi)\), \(r_2(t,\hat\xi)\), and form
        \(\mathcal I_2(t,J_t)\supset\{r_2(t,\xi):\xi\in J_t\}\).
        \If{\(\|(r_1(t,\hat\xi),r_2(t,\hat\xi))\|\le c_r\epsilon\), or
        \(0\notin\mathcal I_2(t,J_t)\), or \(|J_t|\le c_\xi\epsilon\)}
            \State \textbf{break}
        \ElsIf{\(r_1(t,\hat\xi)>0\)}
            \State Set \(\xi^-_t=\hat\xi\).
        \Else
            \State Set \(\xi^+_t=\hat\xi\).
        \EndIf
    \EndLoop
    \If{\(\|(r_1(t,\hat\xi),r_2(t,\hat\xi))\|\le c_r\epsilon\) for
    this \(\hat\xi\)}
        \State \Return \(z(t,\hat\xi)\).
    \ElsIf{\(\mathcal I_2(t,J_t)\subset(0,\infty)\)}
        \State \(t^+=t\).
    \ElsIf{\(\mathcal I_2(t,J_t)\subset(-\infty,0)\)}
        \State \(t^-=t\).
    \Else
        \State \Return \(z(t,\hat\xi)\) \Comment{\(|J_t|\le c_\xi\epsilon\)
        and \(0\in\mathcal I_2(t,J_t)\)}
    \EndIf
\EndWhile
\State Set \(t=\operatorname{mid}([t^-,t^+])\), obtain \(J_t\) by the same sign
expansion, refine it until \(|J_t|\le c_\xi\epsilon\), set
\(\hat\xi=\operatorname{mid}(J_t)\), and return \(z(t,\hat\xi)\).
\end{algorithmic}
\end{algorithm}

\begin{lemma}[Certified subproblem complexity]
\label{lem:bilevel-complexity}
Let the diagonal-metric proximal operator \(\prox_{\eta h}^{D^{-2}}\) be
computable in time \(T_p\).  Suppose Algorithm~\ref{alg:bilevel-certified} is
given an outer bracket of radius \(R_t\) containing \(t^*\), and suppose that,
for every \(t\) in this bracket, the sign expansion terminates with a valid
\(\xi\)-bracket of radius at most \(R_{\xi}\).  Then the algorithm
returns an \(\epsilon\)-accurate scaled proximal point using
\[
    \mathcal O\bigl(\log(R_t/\epsilon)\log(R_{\xi}/\epsilon)\bigr)
\]
residual-oracle evaluations, and therefore costs
\[
    \mathcal O\bigl((d+T_p)
    \log(R_t/\epsilon)\log(R_{\xi}/\epsilon)\bigr).
\]
For fixed problem data this is
\(\mathcal O((d+T_p)\log^2(1/\epsilon))\).  The outer \(t\)-bracket is
supplied by Lemma~\ref{lem:search-rectangle}.  For each trial
\(t\), a valid inner \(\xi\)-bracket is obtained either from a problem-specific
bound or by sign expansion.
\end{lemma}
\begin{proof}
Each residual-oracle evaluation costs \(\mathcal O(d+T_p)\): one
diagonal-metric proximal evaluation and a constant number of inner products.
The sign expansion and the subsequent bisection on the nonincreasing residual
\(r_1(t,\cdot)\) use
\(\mathcal O(\log(R_{\xi}/\epsilon))\) evaluations.  The outer subgradient
bisection contracts its bracket geometrically and uses
\(\mathcal O(\log(R_t/\epsilon))\) certified decisions, each invoking one
inner solve with \(\epsilon_\xi=\Theta(\epsilon)\).  For fixed problem data,
this gives \(\mathcal O(\log^2(1/\epsilon))\) residual-oracle evaluations.
If Algorithm~\ref{alg:bilevel-certified} exits through the small-bracket
branch, then \(r_1(t,\hat\xi)=\mathcal O(\epsilon)\) by the inner sign bracket
and~\eqref{eq:r1-xi-lipschitz}, while
\(|r_2(t,\hat\xi)|=\mathcal O(\epsilon)\) by the enclosure diameter assumption.
Thus the reduced KKT residual is \(\mathcal O(\epsilon)\).  Indeed, with
\(q_\epsilon=q(t,\hat\xi)\), \(z_\epsilon=z(t,\hat\xi)\), and
\(\bar z_\epsilon=X^{-1}q_\epsilon
=z_\epsilon-r_1(t,\hat\xi)D^{-1}u/c\), the exact
\(q\)-stationarity condition gives some \(s_\epsilon\in\partial h(q_\epsilon)\)
such that
\[
    \bar z_\epsilon-x+\eta X^\top s_\epsilon
    =
    r_2(t,\hat\xi)v-\frac{r_1(t,\hat\xi)}{c}D^{-1}u
    \in \partial\Psi_x(\bar z_\epsilon).
\]
The right-hand side has norm \(\mathcal O(\epsilon)\).  Since
\(\Psi_x\) is 1-strongly convex, this subgradient residual gives
\(\|\bar z_\epsilon-z^*\|=\mathcal O(\epsilon)\), and
\(\|z_\epsilon-\bar z_\epsilon\|=\mathcal O(\epsilon)\) gives the same
localization for the returned point.
Multiplying the two logarithmic factors by the per-evaluation cost gives the
stated bound.
\end{proof}

The value of \(T_p\) in Lemma~\ref{lem:bilevel-complexity} depends on the
regularizer and on the diagonal scaling.  For separable regularizers it is
often linear in \(d\): if \(h(x)=\sum_i h_i(x_i)\), then
\[
    \bigl[\prox_{\eta h}^{D^{-2}}(a)\bigr]_i
    =
    \prox_{\eta d_i^2 h_i}(a_i).
\]
For block-separable group regularizers, the same conclusion holds when \(D\) is
constant on each group.  In the scalar case \(D=\alpha I\), the
diagonal-metric proximal operation reduces to an ordinary proximal evaluation
with a rescaled parameter, which is useful for structured nonseparable
penalties whose ordinary proximal mapping is available.

\medskip\noindent\textbf{Safeguarded oracle.}
The oracle used in the DBS iterations first runs the two-dimensional monotone
localization in Algorithm~\ref{alg:monotone-prox}.  If the residual test or
localization certificate is satisfied within a prescribed budget, the localized
point is accepted.  Otherwise the oracle calls the certified bilevel solve in
Algorithm~\ref{alg:bilevel-certified} with the same tolerance.  Thus the
localization phase is the default accelerator, while the bilevel phase supplies
the worst-case certificate.

\begin{proposition}[Safeguarded subproblem complexity]
\label{prop:safeguard}
Run the safeguarded oracle with \(\delta_{\rm loc}=\Theta(\epsilon)\) and
budget
\[
    N_{\rm sg}=\Theta\!\left(
    \log(R_t/\epsilon)\log(R_{\xi}/\epsilon)\right),
\]
and call Algorithm~\ref{alg:bilevel-certified} with the \(t\)-bracket from
Lemma~\ref{lem:search-rectangle} if the localization phase is not certified
within this budget.  Then the oracle returns an \(\epsilon\)-accurate scaled
proximal point using
\[
    \mathcal O\bigl((d+T_p)
    \log(R_t/\epsilon)\log(R_{\xi}/\epsilon)\bigr)
\]
work, hence \(\mathcal O((d+T_p)\log^2(1/\epsilon))\) for fixed problem data.
If the localization phase certifies diameter reduction,
\[
    \operatorname{diam}(\mathcal E_j)
    \le C_{\rm loc}R_0\exp(-c_{\rm loc}j),
\]
for constants \(c_{\rm loc},C_{\rm loc}>0\) with
\(R_0=\operatorname{diam}(\mathcal B_0)\),
then the accelerator phase reaches the localization tolerance before the
fallback is called, and the cost improves to
\(\mathcal O\bigl((d+T_p)\log(1/\epsilon)\bigr)\).
\end{proposition}
\begin{proof}
If the accelerator phase certifies within the budget, the returned point
satisfies the localization criterion and, by~\eqref{eq:z-lipschitz}, is
\(\mathcal O(\epsilon)\)-accurate.  Otherwise the oracle falls back to the
certified bilevel solve.  Lemma~\ref{lem:bilevel-complexity} gives the stated
worst-case bound, and the accelerator budget is of the same order.  Under
diameter reduction, monotonicity keeps every zero of \(\mathcal T\) in the
localization set, and the enclosing diameter reaches \(\Theta(\epsilon)\) after
\(\mathcal O(\log(1/\epsilon))\) accelerator steps.  Each step costs
\(\mathcal O(d+T_p)\), so the fallback is not called and the single-logarithm
bound follows.
\end{proof}

\begin{remark}\label{rem:diameter-localization}
The monotone cuts in Algorithm~\ref{alg:monotone-prox} preserve the zero set,
but a generic enclosing-ellipse update gives volume decrease rather than a
direct Euclidean diameter certificate.  A volume bound alone need not control
the residual at the center when the ellipses become highly eccentric.  This is
why Algorithm~\ref{alg:monotone-prox} is used as an accelerator and the
certified \(\log^2(1/\epsilon)\) bound comes from the bilevel fallback.
Diameter reduction can be enforced by fixed-dimensional cutting-plane schemes
with periodic well-rounding, which keep the localization sets at uniformly
bounded eccentricity
\cite{blandGoldfarbTodd1981ellipsoid,grotschel2012geometric}.
\end{remark}

\medskip\noindent\textbf{Warm-started oracle.}
Inside the outer iteration of Section~\ref{sec:two-sided-dbs}, consecutive
scaled proximal subproblems are strongly correlated.  The DBS oracle therefore
uses the current factor-coordinate image as a predictor and falls back to the
certified solve when the prediction is not certified.  On the inverse side, the
subproblem input is
\(y_k=X_k^{-1}x_k-\eta X_k^\top g_k\), the exact subproblem solution is
\(z_k^*=X_k^{-1}x_{k+1}\), and we use
\[
    z_k^{\rm pred}=X_k^{-1}x_k,
    \qquad
    \delta_k=\|X_k^{-1}s_{k-1}\|
\]
as a predictor and scale estimate.  The corresponding scalar predictions are
\[
    t_k^{\rm pred}=v^\top z_k^{\rm pred},
    \qquad
    \xi_k^{\rm pred}
    =\frac{\langle D^{-1}u,\,y_k-z_k^{\rm pred}\rangle}{c},
\]
with warm radii
\[
    r_t^{\rm ws}=\nu_{\rm ws}\|v\|\delta_k,
    \qquad
    r_\xi^{\rm ws}=\nu_{\rm ws}\|D^{-1}u\|\delta_k/|c|,
\]
where \(\nu_{\rm ws}>1\) is a fixed margin.  The warm phase is used only when
these radii are smaller than the cold-start rectangle of
Lemma~\ref{lem:search-rectangle}.  It first tests the predicted scalar pair;
if the residual is not small enough, Algorithm~\ref{alg:monotone-prox} runs in
the warm region.  The warm region is only a localization scale, not a
certificate: acceptance always requires the residual certificate.  If the warm
budget is exhausted, Algorithm~\ref{alg:bilevel-certified} is called with the
warm brackets, and its sign expansions restore the certified worst-case
guarantee of Proposition~\ref{prop:safeguard}.  The Hessian-side prediction is
analogous, with \(z_k^{\rm pred}=Y_k^\top x_k\).

The localization tolerance also tracks the outer progress:
\[
    \delta_{{\rm loc},k}
    =\max\bigl\{\delta_{\rm abs},\;\nu_{\rm rel}\,\delta_k\bigr\},
\]
with \(\nu_{\rm rel}\in(0,1)\) and \(\delta_{\rm abs}>0\).
Corollary~\ref{cor:relative-localization} shows that this schedule preserves
linear convergence down to the absolute floor when \(\delta_k\) decays
geometrically.  Section~\ref{subsec:slope-warm-outer} reports its numerical
effect.

\subsection{Rank-one secant factorizations}
We next derive rank-one factors of the form~\eqref{eq:factorization}. 
The inverse-side requirement is to satisfy the secant condition
\begin{equation}\label{eq:2.10}
    XX^\top y=s.
\end{equation}

Following classical quasi-Newton update and factorization ideas
\cite{fletcher1970new,goldfarb1976factorized}, we use a factor of the form
\[
X = X_0 + \frac{r\hat{v}^\top}{y^\top X_0\hat{v}},
\]
where 
\[
r = s - X_0X_0^\top y.
\]
Here \(\hat v\) is chosen so that \(y^\top X_0\hat v\neq0\).
Let \(a=y^\top X_0\hat v\).  A direct expansion gives
\[
\begin{aligned}
XX^\top y
&=
\left(X_0+\frac{r\hat v^\top}{a}\right)
\left(X_0^\top y+\frac{\hat v r^\top y}{a}\right)  \\
&=
X_0X_0^\top y + r
+\frac{r^\top y}{a}X_0\hat v
+\frac{(r^\top y)\|\hat v\|^2}{a^2}r  \\
&=
s +
\frac{r^\top y}{a^2}
\left(aX_0\hat v+\|\hat v\|^2r\right).
\end{aligned}
\]
Thus a simple sufficient condition for the secant equation \(XX^\top y=s\) is
\[
r^\top y=0,
\]
which is equivalent to
\[
    s^\top y= y^\top X_0X_0^\top y.
\]

We now take \(X_0=D_I\), where \(D_I\succ0\) is diagonal.  The requirement becomes
\[
    y^\top D_I^2y=s^\top y .
\]
This is a diagonal form of the weak secant equation, a relaxation of the full
secant equation commonly used to obtain positive and inexpensive diagonal
quasi-Newton or spectral-gradient scalings
\cite{dennisWolkowicz1993,barzilai1988two,park2020variable,yu2023mini}.
Substituting this choice into the above construction yields
\[
    X=D_I+\frac{(s-D_I^2y)\hat{v}^\top}{\hat{v}^\top D_Iy},
    \qquad 
    y^\top D_I^2y=s^\top y .
\]
The scalar inverse-side specialization is obtained by
\[
    D_I=\alpha I,\qquad
    \alpha=\sqrt{\frac{s^\top y}{\|y\|^2}} .
\]

Interchanging the roles of \(s\) and \(y\) gives a Hessian-side factor \(Y\)
satisfying \(YY^\top s = y\).  With a diagonal \(D_H\succ0\) satisfying
\[
    s^\top D_H^2s=s^\top y ,
\]
we obtain
\[
Y = D_H + \frac{(y-D_H^2s)\hat v^\top}{\hat v^\top D_Hs}.
\]
The scalar Hessian-side construction is recovered by
\[
    D_H=\beta I,\qquad
    \beta = \sqrt{\frac{s^\top y}{\|s\|^2}}.
\]

The diagonal weak secant equations leave many choices for \(D_I\) and \(D_H\).
For example, if \(\tilde d_i>0\) are arbitrary positive weights, then
\[
    d_{I,i}^2
    =
    \frac{(s^\top y)\tilde d_i^2}{\sum_j \tilde d_j^2y_j^2}
\]
satisfies \(y^\top D_I^2y=s^\top y\).  A parallel formula with \(y_j\) replaced
by \(s_j\) gives a Hessian-side diagonal satisfying \(s^\top D_H^2s=s^\top y\).
One may also project a preferred diagonal scaling onto these weak secant
constraints while imposing lower and upper bounds on the diagonal entries, in
the same spirit as safeguarded diagonal BB and variable-metric constructions
\cite{barzilai1988two,raydan1997barzilai,birgin2000nonmonotone,
park2020variable,yu2023mini}.

The freedom in choosing \(\hat v\) is essential.  We next choose \(\hat v\) in a
diagonal-adapted one-parameter family and show that it recovers the
corresponding diagonally initialized zero-memory Broyden members.

\subsection{Broyden interpretation of the factor coordinate}
\label{sec:broyden-factor}

The factor coordinate also has a matrix-level quasi-Newton interpretation.  The
same rank-one factors that define the scaled proximal oracle generate
diagonally initialized zero-memory Broyden members.
The resulting \(\psi\in[0,1]\) segment is the usual Broyden convex class and
has the corresponding least-change and self-scaling interpretations
\cite{dennisSchnabel1979,orenLuenberger1974,dennisWolkowicz1993}.

Let \(s,y\) be any fixed curvature pair with \(\rho=s^\top y>0\), and suppress
the iteration index.  On the inverse side, let \(D_I\succ0\) satisfy
\(y^\top D_I^2y=\rho\), and define
\[
    H_0=D_I^2,\qquad r=s-H_0y,\qquad \sigma_I=r^\top H_0^{-1}r .
\]
Then \(r^\top y=0\).  For
\[
    \hat v_I(\omega)=D_Iy+\omega D_I^{-1}r,\qquad \omega\ge0,
\]
the factor
\[
    X_I(\omega)=D_I+\frac{r\hat v_I(\omega)^\top}{\rho}
\]
satisfies \(X_I(\omega)X_I(\omega)^\top y=s\), and a direct expansion gives
\[
    X_I(\omega)X_I(\omega)^\top
    =
    H_{\mathrm{DFP}}^D
    +
    \psi_I(\omega)\frac{rr^\top}{\rho},
    \qquad
    \psi_I(\omega)=2\omega+\frac{\sigma_I}{\rho}\omega^2,
\]
where
\[
    H_{\mathrm{DFP}}^D
    =
    H_0-\frac{H_0yy^\top H_0}{\rho}+\frac{ss^\top}{\rho}
\]
is the diagonally initialized zero-memory DFP inverse update.  The
corresponding diagonally initialized BFGS inverse member is
\[
    H_{\mathrm{BFGS}}^D=H_{\mathrm{DFP}}^D+\frac{rr^\top}{\rho}.
\]
Thus, whenever \(0\le\omega\le\omega_I^*\), the parameter
\(\psi_I(\omega)\in[0,1]\) and
\begin{equation}\label{eq:inverse-broyden-convex}
    X_I(\omega)X_I(\omega)^\top
    =
    (1-\psi_I(\omega))H_{\mathrm{DFP}}^D
    +
    \psi_I(\omega)H_{\mathrm{BFGS}}^D .
\end{equation}
This is the diagonally initialized inverse Broyden convex segment, with
\(\omega=0\) giving the DFP endpoint and
\[
    \omega_I^*=\frac{1}{1+\sqrt{1+\sigma_I/\rho}}
\]
giving the BFGS endpoint.

The Hessian side is analogous.  If \(D_H\succ0\) satisfies
\(s^\top D_H^2s=\rho\), and
\[
    B_0=D_H^2,\qquad e=y-B_0s,\qquad \sigma_H=e^\top B_0^{-1}e,
\]
then the factor \(Y_H(\omega)=D_H+e\hat v_H(\omega)^\top/\rho\), with
\(\hat v_H(\omega)=D_Hs+\omega D_H^{-1}e\), satisfies
\[
    Y_H(\omega)Y_H(\omega)^\top
    =
    B_{\mathrm{BFGS}}^D
    +
    \psi_H(\omega)\frac{ee^\top}{\rho},
    \qquad
    \psi_H(\omega)=2\omega+\frac{\sigma_H}{\rho}\omega^2,
\]
where \(B_{\mathrm{BFGS}}^D\) is the diagonally initialized zero-memory BFGS
Hessian update.  Define the corresponding diagonally initialized DFP Hessian
member by
\[
    B_{\mathrm{DFP}}^D
    =
    B_{\mathrm{BFGS}}^D+\frac{ee^\top}{\rho}.
\]
For \(0\le\omega\le\omega_H^*\), where
\[
    \omega_H^*=\frac{1}{1+\sqrt{1+\sigma_H/\rho}},
\]
we have \(\psi_H(\omega)\in[0,1]\) and
\begin{equation}\label{eq:hessian-broyden-convex}
    Y_H(\omega)Y_H(\omega)^\top
    =
    (1-\psi_H(\omega))B_{\mathrm{BFGS}}^D
    +
    \psi_H(\omega)B_{\mathrm{DFP}}^D .
\end{equation}
Thus the Hessian-side factor traces the BFGS-to-DFP Broyden convex segment.

For the scalar specialization, set \(D_I=\alpha I\), \(D_H=\beta I\), and use
the scalar direction
\[
    \hat v(\tau)=\tau s+(1-\tau)y,\qquad \tau\in[0,1],
\]
and define
\begin{equation}\label{eq:scalar-omega-map}
\begin{aligned}
    \omega_I(\tau)
    &=
    \frac{\alpha^2\tau}{1-\tau(1-\alpha^2)},
    \qquad
    \omega_H(\tau)
    &=
    \frac{\beta^2(1-\tau)}{\tau+\beta^2(1-\tau)} .
\end{aligned}
\end{equation}
Then the same scalar direction is collinear with the diagonal-adapted
directions on both sides:
\begin{equation}\label{eq:scalar-collinearity}
\begin{aligned}
    \hat v(\tau)
    &\parallel
    \alpha y+\omega_I(\tau)\alpha^{-1}(s-\alpha^2y),\\
    \hat v(\tau)
    &\parallel
    \beta s+\omega_H(\tau)\beta^{-1}(y-\beta^2s).
\end{aligned}
\end{equation}
Consequently the scalar \(\tau\)-coordinate is just a reparametrization of the
diagonal-adapted \(\omega\)-coordinate in the scalar case.
In particular, \(\tau=0\) gives the inverse-side DFP endpoint and
\(\tau=1\) gives the Hessian-side BFGS endpoint, while
\[
    \tau^*=\frac{\|y\|}{\|s\|+\|y\|}
\]
realizes the scalar self-scaled BFGS inverse member and, on the Hessian side,
the scalar self-scaled DFP member.

\subsection{Self-correcting damping}
\label{sec:damping}
It is necessary to ensure that the pair \((s,y)\) used for updating has positive curvature.
Two widely used techniques in deterministic optimization are skipping
\cite{nocedalWright2006} and damping \cite{powell1978lagrangian}. 
In this work we adopt a self-correcting damping technique in the spirit of
\cite{curtisRobinsonZhou2020}: given parameters
\(0<\theta_1<1<\theta_2\) and the proximal-gradient stepsize \(\eta>0\), we
replace the raw gradient difference \(\bar y\) by
\[
\hat y(\gamma) = \gamma s + (1-\gamma)\eta \bar y .
\]
The scaled gradient displacement \(\eta\bar y\) is used so that the curvature
pair is consistent with the proximal-gradient displacement.  The damping
parameter is
\[
\gamma = \min\Bigl\{\gamma\in[0,1] : 
\theta_1 \le \frac{s^\top\hat y(\gamma)}{\|s\|^2}
\le
\frac{\|\hat y(\gamma)\|^2}{s^\top\hat y(\gamma)}
\le \theta_2\Bigr\}.
\]
The choice \(\gamma=1\) is always feasible, since then \(\hat y=s\) and both
ratios equal one.  The feasible set is closed by continuity and is nonempty.
Therefore the minimum is attained and the damping step is well defined.  After
\(\gamma\) is chosen, set \(\hat y=\hat y(\gamma)\).  The update formulas
impose the secant equations on the
safeguarded pair \((s,\hat y)\).  When the scaled raw pair \((s,\eta\bar y)\) already
satisfies the bounds, no correction is needed.

Unless otherwise stated, \(y\) denotes the safeguarded curvature vector
\(\hat y\).  Thus it
satisfies the uniform curvature bound
\begin{equation}\label{eq:damped-bounds}
\theta_1 \le \frac{s^\top y}{\|s\|^2} \le \frac{\| y\|^2}{s^\top y} \le \theta_2,
\end{equation}
which is essential for constructing the uniform positive definiteness property in the next section.

\section{Two-Sided Decomposed Quasi-Newton Methods}
\label{sec:two-sided-dbs}

We now turn the rank-one factorizations into a family of variable-metric
proximal-gradient methods. The framework has two sides: an inverse-side
construction, which builds \(B_k^{-1}=X_kX_k^\top\), and a Hessian-side
construction, which builds \(B_k=Y_kY_k^\top\).  The preceding section
established the scaled-proximal reduction and identified the Broyden members
represented by the factors.  This section specifies the admissibility
conditions needed for uniform metric bounds, presents a generic DBS algorithmic
framework, and gives the convergence and complexity analysis.

\subsection{Two-sided DBS framework}
\label{sec:two-sided-framework}

Let \(s_k=x_k-x_{k-1}\) and
\(\bar y_k=g_k-g_{k-1}\), where \(g_k=\nabla f(x_k)\).  Let \(y_k\)
denote the damped curvature vector obtained from the pair
\((s_k,\eta\bar y_k)\), so that it satisfies~\eqref{eq:damped-bounds}.  The
identities below are secant equations for this modified curvature pair.  Let
\[
    \rho_k=s_k^\top y_k>0 .
\]
The two sides use the following diagonal-plus-rank-one factors.

\medskip   \noindent\textbf{Inverse side.}
Choose a diagonal \(D_{k,I}\succ0\) satisfying
\[
    y_k^\top D_{k,I}^2y_k=\rho_k,
\]
and set
\begin{equation}\label{eq:inv_factor}
    r_{k,I}=s_k-D_{k,I}^2y_k,\qquad
    X_k = D_{k,I} + \frac{r_{k,I}\hat v_{k,I}^\top}
    {\hat v_{k,I}^\top D_{k,I}y_k},
\end{equation}
where \(\hat v_{k,I}^\top D_{k,I}y_k\ne0\).  Then
\(B_{k,I}^{-1} = X_k X_k^\top\) satisfies
\(B_{k,I}^{-1} y_k = s_k\).

\medskip   \noindent\textbf{Hessian side.}
Choose a diagonal \(D_{k,H}\succ0\) satisfying
\[
    s_k^\top D_{k,H}^2s_k=\rho_k,
\]
and set
\begin{equation}\label{eq:hess-factor}
    e_{k,H}=y_k-D_{k,H}^2s_k,\qquad
    Y_k = D_{k,H} + \frac{e_{k,H}\hat v_{k,H}^\top}
    {\hat v_{k,H}^\top D_{k,H}s_k},
\end{equation}
where \(\hat v_{k,H}^\top D_{k,H}s_k\ne0\).  Then
\(B_{k,H} = Y_k Y_k^\top\) satisfies \(B_{k,H} s_k = y_k\).

The inverse and Hessian constructions impose different secant equations: the
inverse side enforces \(B_{k,I}^{-1}y_k=s_k\), while the Hessian side enforces
\(B_{k,H}s_k=y_k\).  In general these two choices do not produce inverse
matrices of one another.
The update follows the composite proximal gradient step using the corresponding metric:
\[
x_{k+1} = \prox_{\eta h}^{B_k}\!\bigl(x_k - \eta B_k^{-1}\nabla f(x_k)\bigr).
\]
For the inverse side we directly take \(B_k^{-1}=B_{k,I}^{-1}\).  
For the Hessian side we have \(B_k = B_{k,H}\).  Then \(B_k^{-1}\) is computed
through the rank-one inverse of \(Y_k\) as detailed in
Section~\ref{sec:algo-variants}.

\subsection{Admissibility and uniform boundedness}
\label{sec:admissible}
The damping guarantees
\begin{equation}\label{eq:theta-bounds}
\theta_1 \le \frac{s_k^\top y_k}{\|s_k\|^2} \le \frac{\|y_k\|^2}{s_k^\top y_k} \le \theta_2,
\end{equation}
which keeps \(\|s_k\|\) and \(\|y_k\|\) on comparable scales relative to
\(\rho_k=s_k^\top y_k\).

We now specify the admissible direction family for the general diagonal
construction.  On the inverse side we choose
\[
    \hat v_{k,I}
    =
    D_{k,I}y_k+\omega_{k,I}D_{k,I}^{-1}r_{k,I},
\]
and on the Hessian side we choose
\[
    \hat v_{k,H}
    =
    D_{k,H}s_k+\omega_{k,H}D_{k,H}^{-1}e_{k,H},
\]
where \(\omega_{k,I},\omega_{k,H}\ge0\).  These directions are adapted to the
diagonal weak secant equations.  Indeed, since \(r_{k,I}^\top y_k=0\) and
\(e_{k,H}^\top s_k=0\), they satisfy
\[
    \hat v_{k,I}^\top D_{k,I}y_k
    =
    \rho_k,
\]
and
\[
    \hat v_{k,H}^\top D_{k,H}s_k
    =
    \rho_k .
\]
Thus the rank-one denominators in~\eqref{eq:inv_factor}
and~\eqref{eq:hess-factor} are automatically nonzero.

\begin{definition}[Admissible metrics for DBS]\label{def:admissible-dbs}
A metric sequence generated by the DBS framework is admissible if the
safeguarded curvature pairs
satisfy~\eqref{eq:theta-bounds}, the diagonal weak-secant scalings satisfy
\[
    0<d_-\le \lambda_{\min}(D_{k,I}^2),\lambda_{\min}(D_{k,H}^2),
    \qquad
    \lambda_{\max}(D_{k,I}^2),\lambda_{\max}(D_{k,H}^2)\le d_+<\infty,
\]
and the factor directions have the diagonal-adapted form above with
\[
    0\le \omega_{k,I},\omega_{k,H}\le\bar\omega
\]
for constants \(d_-,d_+,\bar\omega\) independent of \(k\).
\end{definition}

The following lemma shows that admissibility is a sufficient condition for
uniform positive definiteness.

\begin{lemma}[Uniform positive definiteness for diagonal factors]
\label{lem:generic-pd}
Assume that the metric sequence generated by DBS is admissible in the sense of
Definition~\ref{def:admissible-dbs}.
Then there exist
constants \(0<m\le M<\infty\), independent of \(k\), such that the inverse-side
and Hessian-side metrics satisfy
\[
    mI\preceq B_{k,I}^{-1}\preceq MI,
    \qquad
    mI\preceq B_{k,H}\preceq MI .
\]
\end{lemma}
\begin{proof}
We prove the inverse-side statement.  The Hessian side is identical after
interchanging \(s_k\) and \(y_k\).  From the damping bounds,
\[
    \|s_k\|\le \rho_k^{1/2}/\sqrt{\theta_1},
    \qquad
    \|y_k\|\le \sqrt{\theta_2\rho_k}.
\]
The diagonal bounds imply
\[
    \|r_{k,I}\|
    \le
    \|s_k\|+\|D_{k,I}^2y_k\|
    \le
    C_r\rho_k^{1/2},
    \qquad
    \|\hat v_{k,I}\|\le C_v\rho_k^{1/2},
\]
where \(C_r,C_v\) depend only on
\(\theta_1,\theta_2,d_-,d_+\), and \(\bar\omega\).  Since
\[
    \hat v_{k,I}^\top D_{k,I}y_k
    =
    y_k^\top D_{k,I}^2y_k+\omega_{k,I}r_{k,I}^\top y_k
    =
    \rho_k ,
\]
we get \(\|X_k\|\le \sqrt{d_+}+C_rC_v\).  For the inverse, the
Sherman--Morrison formula gives
\[
    X_k^{-1}
    =
    D_{k,I}^{-1}
    -
    \frac{D_{k,I}^{-1}r_{k,I}\hat v_{k,I}^\top D_{k,I}^{-1}}
    {\rho_k+\hat v_{k,I}^\top D_{k,I}^{-1}r_{k,I}} .
\]
Moreover,
\[
    \hat v_{k,I}^\top D_{k,I}^{-1}r_{k,I}
    =
    \omega_{k,I} r_{k,I}^\top D_{k,I}^{-2}r_{k,I}\ge0,
\]
so the denominator is at least \(\rho_k\).  The same estimates therefore bound
\(\|X_k^{-1}\|\) uniformly.  Hence
\(X_kX_k^\top\) has eigenvalues bounded above and below uniformly in \(k\).
\end{proof}

The scalar \(\tau\)-family is a special case of the preceding admissible
family.  It writes the auxiliary direction in the original curvature-pair
coordinate
\[
    \hat v_k(\tau)=\tau s_k+(1-\tau)y_k,\qquad \tau\in[0,1].
\]

\begin{corollary}[Admissibility of the scalar \(\tau\)-family]
\label{cor:scalar-admissible}
Assume that the damped pairs satisfy~\eqref{eq:theta-bounds}.  For the scalar
choices
\[
    D_{k,I}=\alpha_k I,\qquad
    \alpha_k^2=\frac{s_k^\top y_k}{\|y_k\|^2},
    \qquad
    D_{k,H}=\beta_k I,\qquad
    \beta_k^2=\frac{s_k^\top y_k}{\|s_k\|^2},
\]
and for any \(\tau_k\in[0,1]\), the scalar direction
\(\hat v_k(\tau_k)=\tau_k s_k+(1-\tau_k)y_k\) is admissible in the sense of
Definition~\ref{def:admissible-dbs}.  Consequently the scalar metrics generated
by DBS are uniformly positive definite.
\end{corollary}
\begin{proof}
The damping bounds imply
\[
    \theta_2^{-1}\le \alpha_k^2\le \theta_1^{-1},
    \qquad
    \theta_1\le \beta_k^2\le \theta_2 .
\]
Thus the scalar diagonals satisfy the diagonal boundedness condition in
Definition~\ref{def:admissible-dbs}.  With the iteration index restored, the
reparametrization~\eqref{eq:scalar-omega-map} and
collinearity~\eqref{eq:scalar-collinearity} give diagonal-adapted parameters
\(\omega_{k,I}(\tau_k),\omega_{k,H}(\tau_k)\in[0,1]\) for every
\(\tau_k\in[0,1]\).  Since the factor formula is invariant under nonzero
rescaling of \(\hat v_k\), the scalar factors are admissible by
Definition~\ref{def:admissible-dbs}.  Lemma~\ref{lem:generic-pd} then gives
uniform positive definiteness.
Thus the scalar \(\tau\)-family is the scalar-coordinate representation of the
same diagonal-adapted \(\omega\)-family.
\end{proof}

\subsection{Algorithmic framework and scalar variants}
\label{sec:algo-variants}

The two-sided factors are applied through rank-one products and rank-one
inverses.  For a diagonal-plus-rank-one matrix
\[
    Y=D+p q^\top,\qquad 1+q^\top D^{-1}p\neq0,
\]
the Sherman--Morrison identity gives
\[
    Y^{-1}
    =
    D^{-1}
    -
    \frac{D^{-1}p q^\top D^{-1}}{1+q^\top D^{-1}p},
    \qquad
    Y^{-\top}
    =
    D^{-1}
    -
    \frac{D^{-1}q p^\top D^{-1}}{1+q^\top D^{-1}p}.
\]
Thus the Hessian-side method applies \(B_k^{-1}\) without forming a dense
inverse.

  For the inverse side, \(B_k^{-1}=X_kX_k^\top\), and the metric proximal step is
  \[
  z_{k+1}
  =
  \prox_{\eta h\circ X_k}
  \bigl(X_k^{-1}x_k-\eta X_k^\top g_k\bigr),
  \qquad
  x_{k+1}=X_kz_{k+1}.
  \]
  For the Hessian side, \(B_k=Y_kY_k^\top\), so \(B_k^{-1}=Y_k^{-\top}Y_k^{-1}\).
  Using the factor \(Y_k^{-\top}\), the step becomes
  \[
  z_{k+1}
  =
  \prox_{\eta h\circ Y_k^{-\top}}
  \bigl(Y_k^\top x_k-\eta Y_k^{-1}g_k\bigr),
  \qquad
  x_{k+1}=Y_k^{-\top}z_{k+1}.
  \]
  In both cases the scaled proximal map is computed by the safeguarded oracle
  of Section~\ref{sub:bilevel-search}.  During the outer iteration this oracle
  is warm-started from the current factor-coordinate image and the previous
  outer step; the two-dimensional monotone localization is used as the
  practical accelerator, and the certified bilevel solve supplies the fallback
  guarantee.  This gives the generic DBS framework summarized in
  Algorithm~\ref{alg:twosided}.

  \begin{algorithm}[htbp!]
  \caption{Generic DBS algorithm for solving \(\min_x f(x)+h(x)\)}
  \label{alg:twosided}
  \begin{algorithmic}[1]
  \Require Initial point \(x_0\), stepsize \(\eta>0\), maximum iteration \(K\),
  damping parameters \(0<\theta_1<1<\theta_2\), side
  \(\in\{\mathrm{Inv},\mathrm{Hess}\}\), diagonal scaling rule,
  factor-coordinate rule, and curvature safeguards
  \(\epsilon_s,\epsilon_\rho>0\), stationarity tolerance
  \(\epsilon_{\rm pg}>0\).
  \State \(g_0=\nabla f(x_0)\).
  \State \(x_1=\prox_{\eta h}(x_0-\eta g_0)\), \(g_1=\nabla f(x_1)\).
  \For{\(k=1,\ldots,K\)}
      \State \(s_k=x_k-x_{k-1}\), \(\bar y_k=g_k-g_{k-1}\).
      \If{\(\|s_k\|\le\epsilon_s\)}
          \State Compute
          \(\mathcal G_\eta(x_k)=\eta^{-1}\bigl(x_k-\prox_{\eta h}(x_k-\eta g_k)\bigr)\).
          \If{\(\|\mathcal G_\eta(x_k)\|\le\epsilon_{\rm pg}\)}
              \State \Return \(x_k\) \Comment{certified approximate stationarity}
          \EndIf
          \State \(x_{k+1}=\prox_{\eta h}(x_k-\eta g_k)\),\quad
          \(g_{k+1}=\nabla f(x_{k+1})\) \Comment{restart with scalar metric}
          \State \textbf{continue}
      \EndIf
      \If{\(s_k^\top(\eta\bar y_k)\le\epsilon_\rho\|s_k\|^2\)}
          \State \(x_{k+1}=\prox_{\eta h}(x_k-\eta g_k)\),\quad
          \(g_{k+1}=\nabla f(x_{k+1})\) \Comment{skip metric update}
          \State \textbf{continue}
      \EndIf
      \State Apply damping to obtain \(y_k\) from \((s_k,\eta\bar y_k)\) such that~\eqref{eq:damped-bounds} holds.
      \If{side \(=\mathrm{Inv}\)}
          \State Choose \(D_{k,I}\succ0\) satisfying
          \(y_k^\top D_{k,I}^2y_k=s_k^\top y_k\), and choose an admissible
          \(\widehat v_{k,I}\).
          \State
          \[
          X_k=D_{k,I}
          +
          \frac{(s_k-D_{k,I}^2y_k)\widehat v_{k,I}^\top}
          {\widehat v_{k,I}^\top D_{k,I}y_k}.
          \]
          \State Compute
          \[
          z_{k+1}
          =
          \prox_{\eta h\circ X_k}
          \bigl(X_k^{-1}x_k-\eta X_k^\top g_k\bigr)
          \]
          by the warm-started safeguarded oracle of
          Section~\ref{sub:bilevel-search}.
          \State \(x_{k+1}=X_kz_{k+1}\).
      \Else
          \State Choose \(D_{k,H}\succ0\) satisfying
          \(s_k^\top D_{k,H}^2s_k=s_k^\top y_k\), and choose an admissible
          \(\widehat v_{k,H}\).
          \State
          \[
          Y_k=D_{k,H}
          +
          \frac{(y_k-D_{k,H}^2s_k)\widehat v_{k,H}^\top}
          {\widehat v_{k,H}^\top D_{k,H}s_k}.
          \]
          \State Compute
          \[
          z_{k+1}
          =
          \prox_{\eta h\circ Y_k^{-\top}}
          \bigl(Y_k^\top x_k-\eta Y_k^{-1}g_k\bigr)
          \]
          by the warm-started safeguarded oracle of
          Section~\ref{sub:bilevel-search}.
          \State \(x_{k+1}=Y_k^{-\top}z_{k+1}\).
      \EndIf
      \State \(g_{k+1}=\nabla f(x_{k+1})\).
  \EndFor
  \Ensure \(x_{K+1}\).
  \end{algorithmic}
  \end{algorithm}

  The concrete scalar variants used in the experiments are listed in
  Table~\ref{tab:dbs-variants}.  They take
  \[
      D_{k,I}=\alpha_k I,\qquad
      \alpha_k=\sqrt{\frac{s_k^\top y_k}{\|y_k\|^2}},
      \qquad
      D_{k,H}=\beta_k I,\qquad
      \beta_k=\sqrt{\frac{s_k^\top y_k}{\|s_k\|^2}},
  \]
  and choose the scalar factor direction
  \[
      \widehat v_k(\tau_k)=\tau_k s_k+(1-\tau_k)y_k .
  \]
  We use the terms \(0\)-DFP and \(0\)-BFGS consistently to denote one-pair,
  zero-memory updates with scalar initialization determined by the current
  curvature pair.  The AdaTau variants use the adaptive self-scaled coordinate
  \(\tau_k=\|y_k\|/(\|s_k\|+\|y_k\|)\).

  \begin{table}[t]
  \centering
  \caption{DBS variants used in the numerical experiments.}
  \label{tab:dbs-variants}
  \small
  \setlength{\tabcolsep}{4pt}
  \begin{tabularx}{\textwidth}{lllX}
  \toprule
  Variant & Side & \(\tau\)-rule & Matrix-level member \\
  \midrule
  DBS-I-0DFP
  & inverse
  & fixed \(0\)
  & \(0\)-DFP inverse, \(H_0=\alpha_k^2 I\) \\

  DBS-H-0BFGS
  & Hessian
  & fixed \(1\)
  & \(0\)-BFGS Hessian, \(B_0=\beta_k^2 I\) \\

  DBS-I-AdaTau
  & inverse
  & \(\tau_k=\frac{\|y_k\|}{\|s_k\|+\|y_k\|}\)
  & self-scaled \(0\)-BFGS inverse, \(H_0=\alpha_k^2 I\) \\

  DBS-H-AdaTau
  & Hessian
  & \(\tau_k=\frac{\|y_k\|}{\|s_k\|+\|y_k\|}\)
  & self-scaled \(0\)-DFP Hessian, \(B_0=\beta_k^2 I\) \\
  \bottomrule
  \end{tabularx}
  \end{table}

 \subsection{Linear convergence under strong convexity}
  \label{sec:conv-rates}

  We next state the convergence guarantee for admissible DBS methods.  We first
  give the exact scaled-proximal oracle result.  Since the safeguarded oracle
  can solve the scaled proximal subproblem to any prescribed accuracy, we then
  record the corresponding finite-accuracy consequence.

  \begin{assumption}[Strong convexity and smoothness]
  \label{ass:strconv}
  The function \(f\) is \(L\)-smooth and \(\mu\)-strongly convex with \(\mu>0\),
  and \(h\) is proper, closed, and convex.
  \end{assumption}

  \begin{lemma}[One-step contraction in a bounded metric]
  \label{lem:metric-one-step}
  Let Assumption~\ref{ass:strconv} hold.  Suppose
  \(mI\preceq B^{-1}\preceq MI\), and define
  \[
      x^+
      =
      \prox_{\eta h}^{B}
      \bigl(x-\eta B^{-1}\nabla f(x)\bigr).
  \]
  If \(0<\eta\le1/(LM)\), then
  \[
      F(x^+)-F(x^*)
      \le
      (1-\eta\mu m)\bigl(F(x)-F(x^*)\bigr).
  \]
  \end{lemma}

  \begin{proof}
  The metric bounds are equivalent to
  \(M^{-1}I\preceq B\preceq m^{-1}I\).  Since \(\eta\le1/(LM)\), we have
  \(f(x^+)\le f(x)+\langle\nabla f(x),x^+-x\rangle
  +(2\eta)^{-1}\|x^+-x\|_B^2\).  Hence by the definition of \(x^+\), we derive 
  \[
      F(x^+)
      \le
      f(x)+\langle\nabla f(x),y-x\rangle
      +\frac{1}{2\eta}\|y-x\|_B^2+h(y), \quad\forall y\in\mathbb{R}^d.
  \]
  Set \(\lambda=\eta\mu m\in[0,1]\) and \(y=(1-\lambda)x+\lambda x^*\).
  Convexity of
  \(h\) and \(\mu\)-strong convexity of \(f\) give
  \[
      f(x)+\langle\nabla f(x),y-x\rangle+h(y)
      \le
      (1-\lambda)F(x)+\lambda F(x^*)
      -\frac{\lambda\mu}{2}\|x-x^*\|^2 .
  \]
  Moreover
  \[
      \frac{1}{2\eta}\|y-x\|_B^2
      \le
      \frac{\lambda^2}{2\eta m}\|x-x^*\|^2
      =
      \frac{\lambda\mu}{2}\|x-x^*\|^2 .
  \]
  The last two displays cancel the distance term and yield
  \(F(x^+)\le(1-\lambda)F(x)+\lambda F(x^*)\), which is the stated
  contraction.
  \end{proof}

  \begin{theorem}[Linear convergence of admissible DBS]
  \label{th:linear}
  Let Assumption~\ref{ass:strconv} hold, and let \(\{x_k\}\) be generated by an
  admissible DBS method.  Suppose that the scaled proximal subproblems are
  solved exactly.  Then there exist constants \(0<m\le M<\infty\), independent of
  \(k\), such that
  \[
      mI\preceq B_k^{-1}\preceq MI,\qquad \forall k.
  \]
  If
  \[
      0<\eta\le \frac{1}{LM},
  \]
  then
  \[
      F(x_k)-F(x^*)
      \le
      (1-\eta\mu m)^k\bigl(F(x_0)-F(x^*)\bigr).
  \]
  \end{theorem}

  \begin{proof}
  The inverse-side bounds follow directly from Lemma~\ref{lem:generic-pd}.  On
  the Hessian side, Lemma~\ref{lem:generic-pd} gives
  \(mI\preceq B_k\preceq MI\), equivalently
  \(M^{-1}I\preceq B_k^{-1}\preceq m^{-1}I\).  Relabeling the reciprocal
  constants gives the inverse-metric bounds stated in the theorem.
  Lemma~\ref{lem:metric-one-step}, applied with \(B=B_k\), gives for
  \(\eta\le1/(LM)\)
  \[
      F(x_{k+1})-F(x^*)
      \le
      (1-\eta\mu m)\bigl(F(x_k)-F(x^*)\bigr).
  \]
  The result is obtained by iteration.
  \end{proof}

  For finite-accuracy implementations, define the local scaled proximal model
  \[
      Q_k(x)
      =
      \langle g_k,x-x_k\rangle
      +
      \frac{1}{2\eta}\|x-x_k\|_{B_k}^2
      +
      h(x).
  \]

  \begin{corollary}[Finite-accuracy DBS]
  \label{cor:finite-accuracy}
  Let the assumptions of Theorem~\ref{th:linear} hold and let
  \(q=1-\eta\mu m\).  Suppose that the computed iterate satisfies the model
  accuracy condition
  \[
      Q_k(x_{k+1})
      \le
      \inf_x Q_k(x)+\varepsilon_k .
  \]
  If \(0<\eta\le 1/(LM)\), then
  \[
      F(x_{k+1})-F(x^*)
      \le
      q\bigl(F(x_k)-F(x^*)\bigr)+\varepsilon_k .
  \]
  Consequently,
  \[
      F(x_K)-F(x^*)
      \le
      q^K\bigl(F(x_0)-F(x^*)\bigr)
      +
      \sum_{j=0}^{K-1} q^{K-1-j}\varepsilon_j .
  \]
  In particular, if
  \[
      q^K\bigl(F(x_0)-F(x^*)\bigr)\le \frac{\epsilon_{\rm out}}{2},
      \qquad
      \varepsilon_j\le \frac{(1-q)\epsilon_{\rm out}}{2},
  \]
  then \(F(x_K)-F(x^*)\le \epsilon_{\rm out}\).
  \end{corollary}

  \begin{proof}
  The proof is the same as that of Theorem~\ref{th:linear}, except that the
  approximate solution of the model \(Q_k\) adds the error term
  \(\varepsilon_k\) to the one-step descent inequality.  Iterating the resulting
  linear recursion gives the stated bound.
  \end{proof}

  The next proposition converts the accuracy certificate returned by the scaled
  proximal oracle into the model accuracy \(\varepsilon_k\) required by
  Corollary~\ref{cor:finite-accuracy}.

  \begin{proposition}[Oracle accuracy to model accuracy]
  \label{prop:loc-to-model}
  Let the assumptions of Theorem~\ref{th:linear} hold, so that
  \(mI\preceq B_k^{-1}\preceq MI\), and let \(x_k^\star=\arg\min_x Q_k(x)\) be
  the exact scaled-proximal step.  Suppose the scaled proximal oracle returns
  \(x_{k+1}\) with a scalar localization certificate
  \(\|(t,\xi)-(t^*,\xi^*)\|\le\delta_{\rm loc}\), and hence,
  by~\eqref{eq:z-lipschitz},
  \(\|x_{k+1}-x_k^\star\|\le C_{X,k}\delta_{\rm loc}\).
  If \(h\) is Lipschitz with constant \(L_{h,k}\) on the segment joining
  \(x_{k+1}\) and \(x_k^\star\), then
  \[
      Q_k(x_{k+1})-\inf_x Q_k(x)
      \le
      C_{M,k}\,\delta_{\rm loc}
      +
      \frac{C_{X,k}^2}{2\eta m}\,\delta_{\rm loc}^2 ,
  \]
  where
  \[
      C_{M,k}
      =
      C_{X,k}\left(
      \left\|g_k+\frac1\eta B_k(x_k^\star-x_k)\right\|
      +L_{h,k}\right).
  \]
  In particular, if these constants are uniformly bounded along the run, then
  choosing \(\delta_{\rm loc}=\mathcal O(\varepsilon_{\rm model})\) yields
  \(\varepsilon_k=Q_k(x_{k+1})-\inf_x Q_k(x)\le\varepsilon_{\rm model}\).
  \end{proposition}
  \begin{proof}
  The bounds \(mI\preceq B_k^{-1}\preceq MI\) give
  \((1/M)I\preceq B_k\), so \(Q_k\) is \((1/(\eta M))\)-strongly convex with
  unique minimizer \(x_k^\star\), and also \(\|B_k\|\le1/m\).  Put
  \(e=x_{k+1}-x_k^\star\).  Expanding the quadratic part of \(Q_k\) around
  \(x_k^\star\) gives
  \[
  \begin{aligned}
      Q_k(x_{k+1})-Q_k(x_k^\star)
      &=
      \left\langle
      g_k+\frac1\eta B_k(x_k^\star-x_k),e
      \right\rangle
      +\frac{1}{2\eta}\|e\|_{B_k}^2  \\
      &\quad +h(x_{k+1})-h(x_k^\star).
  \end{aligned}
  \]
  The local Lipschitz bound on \(h\), the estimate
  \(\|e\|\le C_{X,k}\delta_{\rm loc}\), and \(\|B_k\|\le1/m\) yield the displayed
  value-error estimate.  If the constants in that estimate are uniform, the
    stated choice of \(\delta_{\rm loc}\) gives the model accuracy required by
    Corollary~\ref{cor:finite-accuracy}.
    \end{proof}

    \begin{corollary}[Relative localization schedule]
    \label{cor:relative-localization}
    Let the assumptions of Corollary~\ref{cor:finite-accuracy} and
    Proposition~\ref{prop:loc-to-model} hold, with the constants in
    Proposition~\ref{prop:loc-to-model} uniformly bounded.  Suppose the
    localization tolerance is chosen as
    \[
        \delta_{{\rm loc},k}
        =
        \max\{\delta_{\rm abs},\nu_{\rm rel}\Delta_{k-1}\},
    \]
    where \(\nu_{\rm rel}>0\), \(\delta_{\rm abs}\ge0\), and the predictor scale
    satisfies \(\Delta_{k}\le C_\Delta\sigma^k\) for some
    \(C_\Delta>0\) and \(\sigma\in(0,1)\).  Then the finite-accuracy recursion
    in Corollary~\ref{cor:finite-accuracy} holds with
    \[
        \varepsilon_k
        \le
        \bar C\sigma^k
        +
        C_{\rm abs}(\delta_{\rm abs}+\delta_{\rm abs}^2)
    \]
    for constants \(\bar C,C_{\rm abs}\) independent of \(k\).  Consequently,
   if \(\delta_{\rm abs}=0\), the objective residual still decays
    geometrically, governed by the slower factor among \(q\) and \(\sigma\).
    If \(\delta_{\rm abs}>0\), the same linear
    decay holds down to an
    \(\mathcal O(\delta_{\rm abs}+\delta_{\rm abs}^2)\) accuracy floor.
    \end{corollary}

    \begin{proof}
    Uniform boundedness in Proposition~\ref{prop:loc-to-model} gives
    \(\varepsilon_k\le C_1\delta_{{\rm loc},k}+C_2\delta_{{\rm loc},k}^2\).
    Since \(\delta_{{\rm loc},k}\le
    \delta_{\rm abs}+\nu_{\rm rel}C_\Delta\sigma^{k-1}\), the displayed
    estimate follows after increasing the constants.  Substituting this bound
    into the convolution formula in Corollary~\ref{cor:finite-accuracy} gives a
    weighted sum of two geometric sequences plus the constant floor term.
    \end{proof}

    \subsection{Computational cost}
    \label{sec:overall-complexity}

  Each DBS iteration requires one gradient evaluation, a constant number of
  rank-one matrix-vector products, and one scaled proximal computation.  The
  rank-one linear algebra costs \(\mathcal O(d)\).  Let \(T_g\) denote the
  cost of one gradient evaluation and let \(T_p\) denote the cost of one
  diagonal-metric proximal evaluation required by the scaled proximal solve.
  For localization tolerance \(\delta_{\rm loc}\),
  Proposition~\ref{prop:safeguard} gives the per-iteration cost
  \[
      T_g+
      \mathcal O\bigl((d+T_p)
      \log(R_t/\delta_{\rm loc})
      \log(R_{\xi}/\delta_{\rm loc})\bigr)
  \]
  for the safeguarded oracle, which is
  \(T_g+\mathcal O((d+T_p)\log^2(1/\delta_{\rm loc}))\) for fixed bracket
  radii.  In the accelerator branch it becomes
  \(T_g+\mathcal O((d+T_p)\log(1/\delta_{\rm loc}))\) when the localization
  phase certifies diameter reduction.  In the scalar specialization, \(T_p\) is
  the cost of one ordinary proximal evaluation of \(h\).

  Set \(L_{\rm out}=\log(1/\epsilon_{\mathrm{out}})\).  For a target outer
  accuracy \(\epsilon_{\mathrm{out}}\), the finite-accuracy bound is obtained
  by taking the model error
  \(\varepsilon_{\mathrm{model}}
  =\mathcal O((1-q)\epsilon_{\mathrm{out}})\), where
  \(q=1-\eta\mu m\).  Proposition~\ref{prop:loc-to-model} translates this model
  tolerance into a localization tolerance with the same logarithmic scale.
  With fixed bracket radii, this gives the following simplified outer
  complexity.
  Since Corollary~\ref{cor:finite-accuracy} requires
  \(K=\mathcal O(L_{\rm out})\) outer iterations, the total arithmetic cost is
  \[
      \mathcal O\left(
      \bigl(T_g+(d+T_p)L_{\rm out}^2\bigr)L_{\rm out}
      \right)
  \]
  for the safeguarded oracle, and
  \[
      \mathcal O\left(
      \bigl(T_g+(d+T_p)L_{\rm out}\bigr)L_{\rm out}
      \right)
  \]
  in the same certified accelerator branch.
  If \(T_g=\mathcal O(d)\) and \(T_p=\mathcal O(d)\), these
  simplify to \(\mathcal O\bigl(d\log^3(1/\epsilon_{\mathrm{out}})\bigr)\) and
  \(\mathcal O\bigl(d\log^2(1/\epsilon_{\mathrm{out}})\bigr)\), respectively.

\section{Numerical Experiments}
\label{sec:experiments}

Our experiments test whether diagonal-plus-rank-one Broyden metrics can drive
proximal quasi-Newton steps without letting the scaled proximal subproblem
dominate the per-iteration cost.  The study proceeds at two levels.  The first
isolates the rank-one scaled proximal oracle on SLOPE/OWL subproblems, where
the ordinary proximal map is available but the scaled proximal map is
nontrivial.  The second embeds the same oracle in the full outer algorithm and
measures the iterations, gradient evaluations, and wall-clock time needed to
reach prescribed objective-residual targets.  The outer benchmarks span
SLOPE and group-lasso logistic regression, each on synthetic and real-data
instances, covering both structured nonseparable and
block-separable penalties.

The baselines are chosen to match this oracle model.  FISTA-L and MFISTA-L are
standard accelerated proximal-gradient methods using ordinary proximal
evaluations and Lipschitz-normalized stepsizes.  Prox-BB is a spectral
proximal-gradient method based on Barzilai--Borwein stepsizes and the standard
nonmonotone BB globalization philosophy
\cite{barzilai1988two,raydan1997barzilai,birgin2000nonmonotone}.  The 0SR1
baseline is the zero-memory symmetric-rank-one proximal quasi-Newton method
following Becker and Fadili~\cite{becker2012quasi}.

All methods were implemented in Python 3.12.3.  All experiments in this
section were run on an 8-core Arm Neoverse-N1
(AWS Graviton2) instance with 32 GB RAM using four BLAS threads; all methods
within each table are compared on identical hardware and harness.

\subsection{Experimental Setup}

The training set is denoted by \(\{(a_i,b_i)\}_{i=1}^n\), where
\(a_i\in\mathbb{R}^d\) and \(b_i\in\{-1,1\}\). We solve
\[
    \min_{x\in\mathbb{R}^d} F(x):=f(x)+h(x),
\]
with
\[
    f(x)
    =
    \frac1n\sum_{i=1}^n
    \log\bigl(1+\exp(-b_i a_i^\top x)\bigr)
    +
    \frac{\mu_{\rm r}}{2}\|x\|^2 .
\]
For the SLOPE experiments, the nonsmooth term is the sorted \(\ell_1\)
(SLOPE/OWL) penalty
\[
    h(x)=\sum_{j=1}^d w_j |x|_{(j)},\qquad
    w_1\ge w_2\ge\cdots\ge w_d\ge0,
\]
where \(|x|_{(1)}\ge\cdots\ge |x|_{(d)}\) are the sorted absolute values.  Its
ordinary proximal map is computed by sorting and isotonic regression in
\(\mathcal O(d\log d)\) time.  Thus SLOPE provides a structured test case in
which the ordinary proximal map is efficient, while generalized derivative
information for scaled proximal subproblems is active-block and sorting
dependent.  The group-lasso penalty is specified in
Section~\ref{subsec:group-lasso-logistic}.

The feature matrix is generated from an AR(1) model:
\[
    A_{i1}=\xi_{i1},\qquad
    A_{ij}=\rho A_{i,j-1}+\sqrt{1-\rho^2}\,\xi_{ij},
    \quad j\ge2,
\]
where \(\xi_{ij}\sim N(0,1)\).  The parameter \(\rho\in[0,1)\) controls feature
correlation.  Larger \(\rho\) gives a more ill-conditioned design.  A sparse
ground-truth vector with decreasing nonzero magnitudes is generated, and labels
are formed as
\[
    b_i=\operatorname{sign}
    \left(
    \frac{a_i^\top x_{\rm true}}{\sqrt{|\operatorname{supp}(x_{\rm true})|}}
    +0.35\,\varepsilon_i
    \right),
    \qquad \varepsilon_i\sim N(0,1).
\]
The data are split into \(80\%\) training and \(20\%\) testing, and features are
standardized using the training set statistics.

The SLOPE outer experiments use random seeds
\(\{42,43,44\}\), \(w_0=2\times10^{-4}\) in the
weight sequence \(w_j=w_0\sqrt{2\log(2d/j)}\), and \(\mu_{\rm r}=10^{-3}\).
The outer iteration caps are \(800\) for DBS, 0SR1, and the no-warm-start DBS
ablation, and \(2500\) for the first-order and BB baselines.
The Lipschitz estimate uses \(30\) power iterations, with early termination
only if the iterate norm falls below \(10^{-12}\).

In the outer optimization experiments, DBS, 0SR1, and Prox-BB are
tuned over the tested stepsize grid directly.  For FISTA-L and MFISTA-L, the
same grid is used as a multiplier of the inverse Lipschitz estimate, giving
stepsizes \(\eta/L\), where
\[
    L\approx \frac{\lambda_{\max}(A^\top A)}{4n}+\mu_{\rm r}
\]
is estimated by power iteration.  The scaled proximal subproblems of DBS and
0SR1 are solved by the safeguarded oracle with a budget of \(30\)
two-dimensional localization steps before
invoking the certified bilevel fallback.  In the outer benchmarks, DBS uses
the warm-started oracle of Section~\ref{sub:bilevel-search} with
\(\nu_{\rm ws}=4\), warm budget \(12\), \(\nu_{\rm rel}=0.25\), and
\(\delta_{\rm abs}=10^{-6}\).  The 0SR1 baseline uses the same warm-started
scaled-proximal oracle, giving a matched-oracle comparison of the two metrics.
Wall-clock and proximal-call diagnostics
include all localization and fallback work.

The DBS and 0SR1 metric updates use the same admissibility parameters,
\[
    \theta_1=0.01,\qquad \theta_2=100,
\]
throughout the experiments.  The default outer tests use \(n=10000\) and
\(d=2000\).

Unless otherwise stated, the outer experiments report medians over three
random seeds.  For each \(\rho\) and seed, all methods start from \(x_0=0\) and
use the same generated training/test split.  Each experiment specifies its
tuning grid, and methods with a common tuning parameter use the same grid.  All
reported objective residuals are computed on the training objective.

The value \(F_{\rm ref}\) serves as the best-observed numerical reference for
computing objective residuals.  For each instance, we first run a Prox-BB
reference solve from \(x_0=0\) with initial stepsize \(1\), tolerance
\(10^{-12}\), and a maximum of \(800\) iterations.  For the SLOPE outer
benchmark, where targets down to \(5\times10^{-7}\) are reported, we additionally run
a high-accuracy warm-started DBS solve (\(1200\) outer iterations, absolute
inner floor \(10^{-7}\)) as a second reference.  We then set
\(F_{\rm ref}\) to the smallest of these reference values and the best
value found by any benchmarked run over the tested stepsize grid.  Accordingly, we report
\[
    r_k=[F(x_k)-F_{\rm ref}]_+,
\]
and compare methods by the first iteration \(k\) for which
\(r_k\le\varepsilon_{\rm obj}\).  For each method, seed, and correlation level,
the stepsize is selected from the tested grid by minimizing the number of
iterations needed to reach the target.  If no tested stepsize reaches the
target within the prescribed outer budget, the method is reported as not
reached.  We use the two targets
\(\varepsilon_{\rm obj}\in\{10^{-4},5\times10^{-7}\}\) for the high-correlation
SLOPE outer benchmark, with the stepsize selected separately for each target,
and \(\varepsilon_{\rm obj}=10^{-6}\) for the
group-lasso end-to-end benchmark.  In the target-stopping tables, the column
``Cap'' gives the maximum number of outer iterations allowed for each method;
low-cost first-order and BB baselines are allowed larger caps where indicated.

\subsection{SLOPE Scaled Proximal Oracles}

We first isolate the scaled proximal computation by fixing the metric and
solving the same inverse-side AdaTau SLOPE subproblem for all methods.  The
factor is the one used by the default outer DBS implementation and has the form
\[
    X=\alpha I+uv^\top,\qquad B^{-1}=XX^\top .
\]
The comparison uses the ordinary-proximal-oracle model: each solver may call
the ordinary SLOPE proximal map and evaluate the scaled quadratic model.  The
baselines are FISTA, a strongly convex accelerated proximal-gradient method
(APG-SC), and an inner Prox-BB solver.  All methods start from the same
subproblem center and use a maximum budget of \(600\) inner iterations.  This
is a pointwise subproblem test, so its reference is not the outer objective
reference \(F_{\rm ref}\).  Instead, a high-accuracy run of the safeguarded DBS
oracle supplies a numerical reference point for the scaled proximal map.
Success requires both relative distance at most \(10^{-5}\) to this reference
point and fixed-point residual at most \(10^{-6}\).  The first condition checks
agreement with the high-accuracy solution, while the second checks the
optimality equations directly.  The generic baselines test full-dimensional
first-order inner iterations for the same scaled proximal problem; the DBS
oracle instead uses the factorization to reduce the computation to the
two-dimensional scalar system developed above.

Table~\ref{tab:slope-oracle} reports the representative high-dimensional case
\(d=5000\), with medians over three trials.  The rank-one strength \(r\)
controls the metric conditioning.  In this construction
\(\kappa(B)\approx 4r^2+2\), so the reported instances range from
\(4.0\times 10^4\) to \(4.0\times 10^6\).  The cost column reports ordinary
SLOPE proximal evaluations, the common oracle unit in this experiment;
fixed-point residuals are used only for the common success check above.
DBS-I-AdaTau succeeds in every trial and reaches reference-level accuracy
using \(338\)--\(473\) ordinary proximal evaluations, most of them spent in
the certified bilevel stage that drives the scalar-system residual to the
\(10^{-9}\) level.  The generic accelerated solvers exhaust the
\(600\)-iteration budget and remain several orders of magnitude less accurate
in the fixed-point residual.  Inner-ProxBB is the strongest generic
competitor: with spectral stepsize bounds for the scaled quadratic model it
meets the success thresholds at the two moderate conditionings in about
\(105\)--\(108\) proximal evaluations, but at
\(\kappa(B)\approx4\times10^6\) it fails one trial and its median cost rises
to about \(1200\) proximal evaluations and \(12.6\) seconds, against
\(0.19\) seconds for DBS.  The timing difference is also affected by the
linear algebra: the generic inner solvers apply the dense scaled metric,
whereas DBS uses the diagonal-plus-rank-one factorization.

\begin{table}[t]
\centering
\caption{Scaled proximal oracle comparison for the SLOPE penalty with
\(d=5000\) and the inverse-side AdaTau factor used by the default outer DBS
method. Prox reports ordinary SLOPE proximal evaluations, and Time is reported
in seconds.}
\label{tab:slope-oracle}
\scriptsize
\setlength{\tabcolsep}{2pt}
\begin{tabular}{@{}cclccccc@{}}
\toprule
\(r\) & \(\kappa(B)\) & Method & Succ. & Rel. err. & FP res.
& Prox & Time \\
\midrule
100  & \(4.00{\times}10^4\) & DBS-I-AdaTau & \(3/3\)
& \(4.73{\times}10^{-9}\) & \(5.10{\times}10^{-8}\) & 473 & 0.261 \\
100  & \(4.00{\times}10^4\) & FISTA  & \(0/3\)
& \(2.20{\times}10^{-5}\) & \(1.37{\times}10^{-4}\) & 1202 & 5.359 \\
100  & \(4.00{\times}10^4\) & APG-SC & \(0/3\)
& \(1.98{\times}10^{-5}\) & \(1.23{\times}10^{-4}\) & 1202 & 5.437 \\
100  & \(4.00{\times}10^4\) & Prox-BB & \(3/3\)
& \(6.35{\times}10^{-10}\) & \(3.68{\times}10^{-8}\) & 108 & 1.123 \\
\midrule
300  & \(3.60{\times}10^5\) & DBS-I-AdaTau & \(3/3\)
& \(7.87{\times}10^{-9}\) & \(7.94{\times}10^{-8}\) & 392 & 0.224 \\
300  & \(3.60{\times}10^5\) & FISTA  & \(0/3\)
& \(8.82{\times}10^{-5}\) & \(5.50{\times}10^{-4}\) & 1202 & 5.347 \\
300  & \(3.60{\times}10^5\) & APG-SC & \(0/3\)
& \(7.38{\times}10^{-5}\) & \(4.60{\times}10^{-4}\) & 1202 & 5.307 \\
300  & \(3.60{\times}10^5\) & Prox-BB & \(3/3\)
& \(1.37{\times}10^{-8}\) & \(8.12{\times}10^{-8}\) & 105 & 1.025 \\
\midrule
1000 & \(4.00{\times}10^6\) & DBS-I-AdaTau & \(3/3\)
& \(2.32{\times}10^{-8}\) & \(1.87{\times}10^{-7}\) & 338 & 0.189 \\
1000 & \(4.00{\times}10^6\) & FISTA  & \(0/3\)
& \(9.93{\times}10^{-5}\) & \(6.19{\times}10^{-4}\) & 1202 & 5.260 \\
1000 & \(4.00{\times}10^6\) & APG-SC & \(0/3\)
& \(9.67{\times}10^{-5}\) & \(6.03{\times}10^{-4}\) & 1202 & 5.278 \\
1000 & \(4.00{\times}10^6\) & Prox-BB & \(2/3\)
& \(7.67{\times}10^{-8}\) & \(5.11{\times}10^{-7}\) & 1211 & 12.56 \\
\bottomrule
\end{tabular}
\end{table}

\subsection{SLOPE Logistic Regression Outer Benchmark}
\label{subsec:slope-warm-outer}

We next embed the same SLOPE oracle in the full outer method, which is the
setting that stresses the scaled proximal oracle: the ordinary SLOPE proximal
map is cheap for first-order methods, while each DBS step must solve a scaled
SLOPE proximal subproblem through the two-dimensional localization oracle.
The benchmark compares DBS-I-AdaTau, run with its default warm-started
oracle, against the 0SR1, FISTA-L, MFISTA-L, and Prox-BB
baselines on the two highly correlated regimes \(\rho\in\{0.99,0.995\}\),
whose design condition numbers are about \(10^5\) across seeds.  The
rank-one 0SR1 baseline runs under the same outer iteration cap and solves
its scaled proximal subproblems through the same warm-started oracle, so the
DBS--0SR1 comparison isolates the quality of the metric.  We also run a
no-warm-start DBS ablation with the same \(800\)-iteration cap, the same
stepsize grid, and a fixed cold-oracle residual tolerance \(10^{-4}\).  The
stepsize
grid is \(\eta\in\{0.1,0.5,1,2,5\}\) for DBS and 0SR1,
\(\eta\in\{0.5,1,2\}\) as a multiplier of \(1/L\) for FISTA-L and MFISTA-L,
and \(\eta\in\{0.5,1,2\}\) for Prox-BB.  For each method, seed, \(\rho\),
and target, the stepsize is selected by the target-stopping rule described in
the setup.

Table~\ref{tab:slope-target-accuracy} reports the target-accuracy results at
\(\varepsilon_{\rm obj}=10^{-4}\) and \(5\times10^{-7}\), and
Figure~\ref{fig:slope-warm-gap-time} shows the corresponding median residual
trajectories in wall-clock time.  The main observations are:
\begin{itemize}
    \item Warm starting cuts the per-iteration scaled-proximal work by three-
    to fivefold.  Under the same \(800\)-iteration cap, the cold ablation
    spends about \(198\) and \(186\) ordinary proximal evaluations per
    iteration at \(\rho=0.99\) and \(0.995\), whereas the selected
    warm-started DBS runs spend about \(39\) and \(51\).  The cold ablation
    reaches the moderate target but misses \(5\times10^{-7}\) on every seed.
    \item Compared with first-order methods, DBS uses far fewer gradients at
    the stringent target.  FISTA-L reaches \(5\times10^{-7}\) on all seeds but
    needs median iteration counts \(1946\) and \(1957\), compared with
    \(433\) and \(750\) for DBS.  DBS is faster at \(\rho=0.99\) and
    comparable at \(\rho=0.995\); at the moderate target \(10^{-4}\), the
    cheaper FISTA-L iterations remain faster in wall-clock time.
    \item The 0SR1 comparison shows the benefit of the DBS metric under the
    same warm-started oracle.  Although the per-iteration proximal costs are
    similar, 0SR1 needs more iterations at reached targets and misses the
    tight target at \(\rho=0.995\), while DBS reaches every seed.
\end{itemize}
Overall, the SLOPE outer test shows that the warm-started oracle keeps the
scaled-proximal work controlled and turns the stronger DBS metric into a
large gradient-evaluation advantage.

\begin{table}[t]
\centering
\caption{Target-accuracy SLOPE logistic comparison on highly correlated
instances, for targets \(\varepsilon_{\rm obj}=10^{-4}\) (top block) and
\(5\times10^{-7}\) (bottom block).  For each method, seed, and target, the stepsize
is selected by the target-stopping rule described in the setup.  Cap is the
maximum number of outer iterations allowed.  Iter. and Target time are medians
over successful seeds.  SProx/it. is the median number of ordinary SLOPE
proximal evaluations used per outer iteration inside the scaled-proximal
oracle; DBS-cold is the no-warm-start DBS ablation.  Budget time is the median
total time of the selected runs, including full-budget runs for unsuccessful
seeds.}
\label{tab:slope-target-accuracy}
\scriptsize
\setlength{\tabcolsep}{3pt}
\begin{tabular}{@{}clcccccc@{}}
\toprule
\(\rho\) & Method & Cap & Succ. & Iter. & Target (s) & SProx/it. & Budget (s) \\
\midrule
\multicolumn{8}{@{}l}{\emph{Target \(\varepsilon_{\rm obj}=10^{-4}\)}}\\
\midrule
0.99 & DBS-I-AdaTau & 800 & \(3/3\) & 203 & 6.08 & 38.6 & 13.9 \\
0.99 & DBS-cold & 800 & \(3/3\) & 179 & 6.93 & 198.0 & 32.4 \\
0.99 & 0SR1 & 800 & \(3/3\) & 338 & 6.60 & 37.0 & 13.2 \\
0.99 & FISTA-L & 2500 & \(3/3\) & 371 & 2.77 & -- & 18.2 \\
0.99 & MFISTA-L & 2500 & \(3/3\) & 371 & 3.25 & -- & 25.3 \\
0.99 & Prox-BB & 2500 & \(3/3\) & 837 & 9.04 & -- & 27.3 \\
\midrule
0.995 & DBS-I-AdaTau & 800 & \(3/3\) & 299 & 7.63 & 51.0 & 16.4 \\
0.995 & DBS-cold & 800 & \(3/3\) & 228 & 8.25 & 185.7 & 30.0 \\
0.995 & 0SR1 & 800 & \(3/3\) & 544 & 10.96 & 53.6 & 16.5 \\
0.995 & FISTA-L & 2500 & \(3/3\) & 478 & 3.74 & -- & 17.7 \\
0.995 & MFISTA-L & 2500 & \(3/3\) & 478 & 4.35 & -- & 24.9 \\
0.995 & Prox-BB & 2500 & \(3/3\) & 1327 & 14.53 & -- & 26.3 \\
\midrule
\multicolumn{8}{@{}l}{\emph{Target \(\varepsilon_{\rm obj}=5\times10^{-7}\)}}\\
\midrule
0.99 & DBS-I-AdaTau & 800 & \(3/3\) & 433 & 8.94 & 38.6 & 13.9 \\
0.99 & DBS-cold & 800 & \(0/3\) & -- & -- & 195.3 & 31.7 \\
0.99 & 0SR1 & 800 & \(3/3\) & 744 & 12.75 & 37.0 & 13.2 \\
0.99 & FISTA-L & 2500 & \(3/3\) & 1946 & 14.10 & -- & 18.2 \\
0.99 & MFISTA-L & 2500 & \(2/3\) & 1654.5 & 16.12 & -- & 25.3 \\
0.99 & Prox-BB & 2500 & \(0/3\) & -- & -- & -- & 27.3 \\
\midrule
0.995 & DBS-I-AdaTau & 800 & \(3/3\) & 750 & 15.57 & 51.0 & 16.4 \\
0.995 & DBS-cold & 800 & \(0/3\) & -- & -- & 189.8 & 31.2 \\
0.995 & 0SR1 & 800 & \(0/3\) & -- & -- & 53.6 & 16.5 \\
0.995 & FISTA-L & 2500 & \(3/3\) & 1957 & 14.17 & -- & 17.7 \\
0.995 & MFISTA-L & 2500 & \(0/3\) & -- & -- & -- & 24.9 \\
0.995 & Prox-BB & 2500 & \(0/3\) & -- & -- & -- & 28.0 \\
\bottomrule
\end{tabular}
\end{table}

\begin{figure}[t]
\centering
\includegraphics[width=\textwidth]{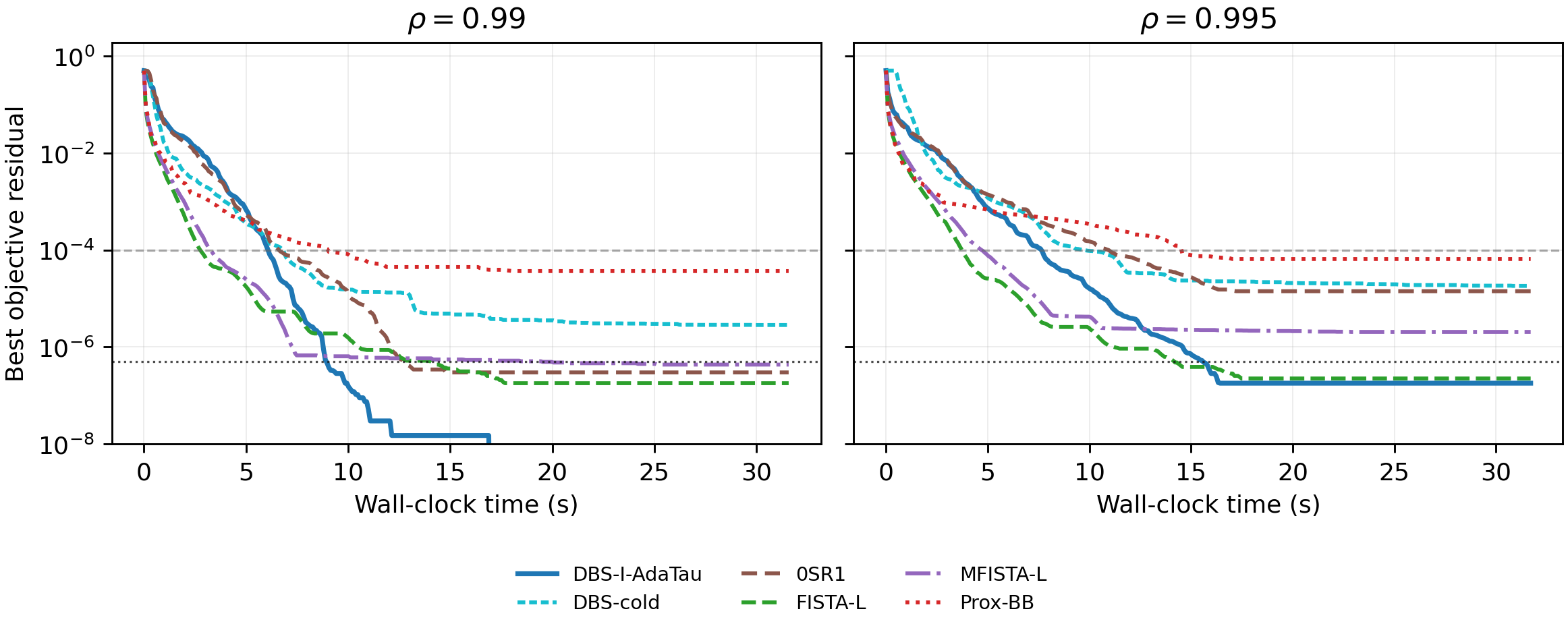}
\caption{Median best objective residual versus wall-clock time for the SLOPE
outer benchmark, using the same per-seed stepsize selections as the
\(5\times10^{-7}\) block of Table~\ref{tab:slope-target-accuracy}, including
the cold-start ablation.  The dashed and dotted horizontal lines mark
\(10^{-4}\) and \(5\times10^{-7}\), respectively.  Residuals are clipped at
\(10^{-8}\): below this level they reflect the resolution of the best-observed
reference rather than additional accuracy.}
\label{fig:slope-warm-gap-time}
\end{figure}

\medskip\noindent\textbf{Real-data SLOPE check.}
To check whether the gradient-evaluation advantage translates into time when
the smooth-loss pass is more expensive, we also run a single-seed real-data
SLOPE experiment on \texttt{real-sim}.  We use \(30000\) training examples,
the \(500\) most frequent base features, and augmented each base feature with
three noisy correlated copies (\(d=2000\), copy noise \(0.01\)).  The target is
\(F(x_k)-F_{\rm ref}\le5\times10^{-7}\); the estimated design condition
number is about \(2\times10^6\).  All methods follow the protocol of the
synthetic SLOPE benchmark under a common \(300\)-iteration budget, and the
stepsize is selected by the target-stopping rule described in the setup.
Table~\ref{tab:real-slope-flagship} shows that
DBS-I-AdaTau reaches the target in \(168\) iterations and \(7.99\) seconds.
The shared-oracle 0SR1 baseline needs \(255\) iterations and \(11.42\)
seconds, MFISTA-L needs \(292\) iterations and \(13.19\) seconds, while
FISTA-L and Prox-BB do not reach the target within the iteration
budget.  Thus, on this larger real-data instance, the iteration advantage
observed in the synthetic SLOPE benchmark becomes a clear wall-clock advantage.

\begin{table}[t]
\centering
\caption{Real-data SLOPE logistic regression on \texttt{real-sim} with
\(n=30000\), \(d=2000\), and condition estimate about \(2\times 10^6\).  The
target is \(F(x_k)-F_{\rm ref}\le5\times10^{-7}\); all methods share a
\(300\)-iteration budget and a common best-observed reference.}
\label{tab:real-slope-flagship}
\scriptsize
\setlength{\tabcolsep}{3pt}
\begin{tabular}{@{}lccccc@{}}
\toprule
Method & Best \(\eta\) & Iter. & Target (s) & Final gap & Budget (s) \\
\midrule
DBS-I-AdaTau & 2.0 & 168 & 7.99 & \(5.96{\times}10^{-8}\) & 13.13 \\
0SR1 & 1.0 & 255 & 11.42 & \(1.04{\times}10^{-7}\) & 13.16 \\
FISTA-L & 2.0 & -- & -- & \(1.41{\times}10^{-5}\) & 10.37 \\
MFISTA-L & 2.0 & 292 & 13.19 & \(2.24{\times}10^{-7}\) & 13.56 \\
Prox-BB & 0.5 & -- & -- & \(7.61{\times}10^{-6}\) & 20.62 \\
\bottomrule
\end{tabular}
\end{table}

\subsection{Group-Lasso Logistic Regression}
\label{subsec:group-lasso-logistic}

We finally evaluate the full DBS outer method on logistic regression with a
nonoverlapping group-lasso penalty,
\[
    h(x)=\lambda_{\rm g}\sum_{g\in\mathcal G}\|x_g\|_2 .
\]
This end-to-end target-accuracy benchmark tests the method on a second
penalty family, in the regime where the scaled proximal step is cheap.
The ordinary proximal map is block soft-thresholding over the groups,
so the scalar DBS specialization again uses only ordinary proximal evaluations,
but each scaled proximal step is considerably cheaper than in the SLOPE case.
The synthetic instances use \(n=3000\), \(d=500\), group size \(20\), six
active groups, \(\lambda_{\rm g}=10^{-3}\), and \(\mu_{\rm r}=10^{-3}\).  We
use the same three seeds \(\{42,43,44\}\), a maximum outer budget \(K=100\),
and the stepsize grid \(\eta\in\{0.5,1,2\}\) for all methods.

Table~\ref{tab:group-lasso-outer} reports the target-accuracy results for
\(\rho=0.9\) and \(\rho=0.99\) under the target \(r_k\le10^{-6}\).  The
low-cost FISTA-L and Prox-BB baselines are run with a larger
\(3000\)-iteration cap, so failure means that no tested stepsize reaches the
same objective target within that extended budget.  The 0SR1 baseline is
reported with the same target-stopping quantities under its \(100\)-iteration
cap.  DBS-I-AdaTau reaches the target on every seed in both regimes and uses
an order of magnitude fewer iterations and gradient evaluations than FISTA-L
(median \(70\) versus \(672\) at \(\rho=0.9\) and \(86\) versus \(770\) at
\(\rho=0.99\)) at comparable median target time (\(2.07\) versus \(1.61\)
seconds and \(1.91\) versus \(1.97\) seconds), showing that the adaptive
inverse-side DBS metric converts its iteration advantage into competitive
end-to-end time when the ordinary proximal map is inexpensive.  0SR1, run
with the same warm-started oracle, reaches the target on all three seeds at
\(\rho=0.9\), in more iterations than DBS (median \(95\)), and on two of
three seeds at \(\rho=0.99\), where the remaining seed levels off at
\(4.7\times10^{-6}\).  Prox-BB reaches the target quickly on the successful
seed, but succeeds on only one of three seeds in each correlation regime.

\begin{table}[t]
\centering
\caption{Target-accuracy group-lasso logistic comparison.  For each method and
seed, the stepsize is selected by the target-stopping rule described in the
setup.  Cap is the maximum number of outer iterations allowed in the
target-stopping run.  Iter. and Target time are medians over successful seeds.
Budget time is the median total time of the selected runs, including
full-budget runs for unsuccessful seeds.}
\label{tab:group-lasso-outer}
\scriptsize
\setlength{\tabcolsep}{3pt}
\begin{tabular}{@{}clccccc@{}}
\toprule
\(\rho\) & Method & Cap & Succ. & Iter. & Target (s) & Budget (s) \\
\midrule
0.9 & DBS-I-AdaTau & 100 & \(3/3\) & 70.0 & 2.07 & 2.46 \\
0.9 & 0SR1 & 100 & \(3/3\) & 95.0 & 2.24 & 2.28 \\
0.9 & FISTA-L & 3000 & \(3/3\) & 672.0 & 1.61 & 7.19 \\
0.9 & Prox-BB & 3000 & \(1/3\) & 127.0 & 0.39 & 9.05 \\
\midrule
0.99 & DBS-I-AdaTau & 100 & \(3/3\) & 86.0 & 1.91 & 2.06 \\
0.99 & 0SR1 & 100 & \(2/3\) & 84.5 & 1.77 & 2.05 \\
0.99 & FISTA-L & 3000 & \(3/3\) & 770.0 & 1.97 & 7.27 \\
0.99 & Prox-BB & 3000 & \(1/3\) & 189.0 & 0.61 & 8.81 \\
\bottomrule
\end{tabular}
\end{table}

\medskip\noindent\textbf{Real-data group-lasso check.}
As a real-data counterpart, we also run a single-seed experiment on
\texttt{real-sim} with the same
grouped-copy construction as the real-data SLOPE check, but with a
group-lasso penalty.  To obtain a controlled grouped-copy instance, we keep
all available training examples after the train-test split (\(n=57847\)),
retain the \(500\) most frequent base features, and augment each base feature
with three noisy correlated copies.  The resulting problem has \(d=2000\),
\(500\) groups of size \(4\), copy noise \(0.01\),
\(\lambda_{\rm g}=10^{-3}\), \(\mu_{\rm r}=10^{-3}\), and condition estimate
about \(2.4\times10^6\).  Each group contains one base feature and its three
copies, so the block penalty is aligned with the correlation structure of the
data.  The target is \(F(x_k)-F_{\rm ref}\le10^{-6}\).  DBS and 0SR1 use a
\(300\)-iteration cap, while FISTA-L, MFISTA-L, and Prox-BB use a
\(1000\)-iteration cap.  The stepsize is selected by the target-stopping rule
described in the setup.

Table~\ref{tab:real-group-lasso} shows that DBS-I-AdaTau reaches the target
in \(159\) iterations and \(31.02\) seconds, faster than 0SR1, FISTA-L,
MFISTA-L, and Prox-BB.  MFISTA-L is the closest time competitor, reaching the
target in \(31.58\) seconds but requiring \(316\) iterations; FISTA-L needs
\(672\) iterations and \(51.79\) seconds.  This real-data check complements
the synthetic group-lasso results by showing that the iteration advantage can
also translate into the fastest target time for a block-separable penalty.

\begin{table}[t]
\centering
\caption{Real-data group-lasso logistic regression on \texttt{real-sim} with
\(n=57847\), \(d=2000\), \(500\) groups of size \(4\), and condition estimate
about \(2.4\times10^6\).  The target is
\(F(x_k)-F_{\rm ref}\le10^{-6}\); the reference is the minimum of a Prox-BB
reference solve and the best value observed by any run.  Final gap is measured
at the iteration cap, so for nonmonotone methods it may exceed a target reached
earlier.}
\label{tab:real-group-lasso}
\scriptsize
\setlength{\tabcolsep}{3pt}
\begin{tabular}{@{}lccccc@{}}
\toprule
Method & Best \(\eta\) & Iter. & Target (s) & Final gap & Budget (s) \\
\midrule
DBS-I-AdaTau & 1.0 & 159 & 31.02 & \(7.45{\times}10^{-8}\) & 45.70 \\
0SR1 & 5.0 & 228 & 38.65 & \(8.94{\times}10^{-8}\) & 44.66 \\
FISTA-L & 2.0 & 672 & 51.79 & \(1.09{\times}10^{-6}\) & 77.24 \\
MFISTA-L & 2.0 & 316 & 31.58 & \(1.04{\times}10^{-7}\) & 109.56 \\
Prox-BB & 2.0 & 267 & 33.26 & \(3.43{\times}10^{-7}\) & 126.90 \\
\bottomrule
\end{tabular}
\end{table}

\subsection{Summary}

The experiments support three complementary conclusions.  First, the scalar
specialization turns the diagonal-plus-rank-one factorization into a practical
ordinary-proximal oracle: it solves high-dimensional SLOPE scaled proximal
subproblems of condition number up to \(4.0\times 10^6\), reaching certified
reference-level accuracy with a few hundred ordinary proximal evaluations.
This shows that Broyden-type rank-one
factorizations can be paired with structured nonsmooth penalties through
ordinary proximal evaluations in the scalar setting.

Second, the SLOPE outer benchmark shows how the same oracle behaves inside the
full method for sorted \(\ell_1\) regularization.  The warm start built into
the oracle, together with
the relative localization tolerance, keeps the selected runs at about
\(39\)--\(51\) ordinary proximal evaluations per outer iteration on average,
versus about \(186\)--\(198\) for the same-cap cold-started ablation.
As a consequence, DBS-I-AdaTau reaches
the stringent target \(5\times10^{-7}\) on every SLOPE outer seed with far fewer
outer iterations and gradient evaluations than FISTA-L, while the cold
ablation reaches no seed at this target and stalls at the finite-accuracy
floor set by the fixed inner tolerance.  DBS remains competitive in
wall-clock time.  The rank-one 0SR1 baseline, run with
the same warm-started oracle at matched per-iteration cost, needs roughly
\(1.7\)--\(1.8\) times as many outer iterations at every reached target and misses
the tight target in the harder regime, isolating the value of the rank-two
Broyden-class metric.  On the larger \texttt{real-sim} SLOPE check, the same
target is reached in \(168\) iterations and \(7.99\) seconds,
faster than 0SR1 and MFISTA-L, while FISTA-L and Prox-BB do not reach the
target within the common iteration budget.

Finally, the group-lasso benchmark shows the end-to-end value of the adaptive
inverse-side metric when the ordinary proximal map is inexpensive.
DBS-I-AdaTau reaches \(r_k\le10^{-6}\) on every seed, uses an order of magnitude
fewer iterations than FISTA-L at comparable median target time, and is
more reliable than Prox-BB on the tested correlated synthetic instances.  On
the real-data grouped-copy instance, the same method reaches the target in
the fewest iterations and the shortest target time among all tested methods.

\section{Conclusions}
\label{sec:conclusions}

This paper developed diagonal-plus-rank-one factorizations for
diagonally initialized zero-memory quasi-Newton metrics in composite
optimization. The diagonal initial scalings satisfy weak secant equations, and
the inverse-side and Hessian-side factors impose the corresponding secant
conditions for the damped curvature pair while recovering diagonally initialized
zero-memory DFP/BFGS-type Broyden members at the matrix level. The resulting DBS
procedure exploits the factor structure to compute the scaled proximal step
by combining a two-dimensional monotone localization accelerator with a
certified bilevel solve, both using diagonal-metric proximal evaluations. The scalar implementation studied numerically is
recovered by taking the diagonal scaling to be a multiple of the identity, in
which case the inner oracle requires only ordinary proximal evaluations of the
nonsmooth term.  In this scalar specialization, DBS gives a derivative-free,
ordinary-proximal-oracle route
to diagonally initialized zero-memory Broyden-convex members generated by
diagonal weak-secant scalings, while preserving the single-proximal-evaluation
oracle structure of the symmetric rank-one method in a two-dimensional search.

The analysis established admissibility conditions, uniform positive
definiteness, linear convergence under strong convexity, and the corresponding
complexity bounds for the two-sided DBS methods.  The numerical results show
three practical consequences of the factorization.  On SLOPE/OWL, DBS provides
an effective scaled-proximal oracle using only ordinary proximal evaluations.
On the SLOPE outer benchmark, the warm-started oracle, which predicts the
scalar pair of the next scaled proximal point from the factor image of the
current iterate and tracks the outer progress with a relative localization
tolerance, keeps the scaled-proximal overhead low and exposes the
iteration and gradient-evaluation advantage of the adaptive metric at tight
target accuracies; on the larger \texttt{real-sim} SLOPE check this advantage
also appears directly in wall-clock target time.  On the group-lasso
target-accuracy benchmarks, the adaptive
inverse-side metric reaches the prescribed objective residual reliably, with
far fewer outer iterations at comparable median target time relative to
Lipschitz-normalized FISTA on the synthetic correlated instances and the
fastest target time on the real-data grouped-copy instance.  The metric
ablation further shows that the inverse-side adaptive variant gives the
strongest target-reaching behavior among the tested DBS choices.

Several extensions remain natural.  First, the scalar specialization studied
here should be compared further with fully diagonal and block-diagonal
initializations when the diagonal-metric proximal map is also inexpensive.
Promising choices include projected BB-type or coordinatewise curvature
estimates with explicit conditioning safeguards
\cite{barzilai1988two,park2020variable,yu2023mini}, blockwise scalings aligned
with active-set structure, and adaptive rules that fall back to the scalar
metric unless the richer scaling predicts a clear reduction in the
proximal-gradient residual.  Second, the proximal-oracle nature of DBS makes it
attractive for regularizers whose ordinary proximal mappings are available but
whose generalized Jacobians are cumbersome, such as sorted \(\ell_1\) penalties,
fused lasso, total-variation penalties, overlapping group regularizers, and
spectral regularizers.  These extensions would further clarify the trade-off
between specialized semismooth Newton solvers, derivative-free inner solvers,
and factorization-based monotone proximal searches.  Third, the outer
iteration advantage observed here should be tested on models where each
gradient or smooth-loss evaluation is substantially more expensive, such as
large-scale or matrix-structured SLOPE variants.  In such regimes, the reduction
in gradient evaluations may translate more directly into wall-clock gains.

\section*{Acknowledgments}
The author is grateful to Dr. Xiaoyu Wang for sharing experimental data and code,
and to Mr. Yichuan Cao for providing computational resources.  This work was
conducted at the State Key Laboratory of Scientific and Engineering Computing
(LSEC), whose research environment is appreciated.

\section*{Disclosure statement}
No potential conflict of interest was reported by the authors.

\section*{Funding}
This work was partially funded by the National Natural Science Foundation of
China under grant No.~12288201.

\bibliographystyle{tfs}
\bibliography{references}

@inproceedings{andrew2007scalable,
  author    = {Andrew, Galen and Gao, Jianfeng},
  title     = {Scalable training of {L1}-regularized log-linear models},
  booktitle = {Proceedings Of The 24th International Conference On Machine Learning},
  pages     = {33--40},
  publisher = {ACM},
  year      = {2007},
  doi       = {10.1145/1273496.1273501}
}

@article{barzilai1988two,
  author  = {Barzilai, Jonathan and Borwein, Jonathan M.},
  title   = {Two-point step size gradient methods},
  journal = {IMA Journal Of Numerical Analysis},
  volume  = {8},
  number  = {1},
  pages   = {141--148},
  year    = {1988},
  doi     = {10.1093/imanum/8.1.141}
}

@book{bauschkeCombettes2017,
  author    = {Bauschke, Heinz H. and Combettes, Patrick L.},
  title     = {Convex analysis and monotone operator theory in Hilbert spaces},
  series    = {CMS Books In Mathematics},
  edition   = {2},
  publisher = {Springer},
  address   = {Cham},
  year      = {2017},
  doi       = {10.1007/978-3-319-48311-5}
}

@article{beck2009fast,
  author  = {Beck, Amir and Teboulle, Marc},
  title   = {A fast iterative shrinkage-thresholding algorithm for linear inverse problems},
  journal = {SIAM Journal On Imaging Sciences},
  volume  = {2},
  number  = {1},
  pages   = {183--202},
  year    = {2009},
  doi     = {10.1137/080716542}
}

@inproceedings{becker2012quasi,
 author = {Becker, Stephen and Fadili, Jalal},
 booktitle = {Advances in Neural Information Processing Systems},
 editor = {F. Pereira and C.J. Burges and L. Bottou and K. Weinberger},
 pages = {2618--2626},
 publisher = {Curran Associates, Inc.},
 title = {A quasi-Newton proximal splitting method},
 volume = {25},
 year = {2012}
}

@article{becker2011templates,
  author  = {Becker, Stephen R. and Cand{\`e}s, Emmanuel J. and Grant, Michael C.},
  title   = {Templates for convex cone problems with applications to sparse signal recovery},
  journal = {Mathematical Programming Computation},
  volume  = {3},
  number  = {3},
  pages   = {165--218},
  year    = {2011},
  doi     = {10.1007/s12532-011-0029-5}
}

@article{becker2019quasi,
  author  = {Becker, Stephen and Fadili, Jalal and Ochs, Peter},
  title   = {On quasi-Newton forward-backward splitting: proximal calculus and convergence},
  journal = {SIAM Journal On Optimization},
  volume  = {29},
  number  = {4},
  pages   = {2445--2481},
  year    = {2019},
  doi     = {10.1137/18M1167152}
}

@article{birgin2000nonmonotone,
  author  = {Birgin, Ernesto G. and Mart{\'i}nez, Jos{\'e} Mario and Raydan, Marcos},
  title   = {Nonmonotone spectral projected gradient methods on convex sets},
  journal = {SIAM Journal On Optimization},
  volume  = {10},
  number  = {4},
  pages   = {1196--1211},
  year    = {2000},
  doi     = {10.1137/S1052623497330963}
}

@article{bonettini2016extrapolation,
  author  = {Bonettini, Silvia and Porta, Federica and Ruggiero, Valeria},
  title   = {A variable metric forward-backward method with extrapolation},
  journal = {SIAM Journal On Scientific Computing},
  volume  = {38},
  number  = {4},
  pages   = {A2558--A2584},
  year    = {2016},
  doi     = {10.1137/15M1025098}
}

@article{bonettini2016variable,
  author  = {Bonettini, Silvia and Loris, Ignace and Porta, Federica and Prato, Marco},
  title   = {Variable metric inexact line-search-based methods for nonsmooth optimization},
  journal = {SIAM Journal On Optimization},
  volume  = {26},
  number  = {2},
  pages   = {891--921},
  year    = {2016},
  doi     = {10.1137/15M1019325}
}

@incollection{combettes2011proximal,
  author    = {Combettes, Patrick L. and Pesquet, Jean-Christophe},
  title     = {Proximal splitting methods in signal processing},
  booktitle = {Fixed-Point Algorithms For Inverse Problems In Science And Engineering},
  editor    = {Bauschke, Heinz H. and Burachik, Regina S. and Combettes, Patrick L. and Elser, Veit and Luke, D. Russell and Wolkowicz, Henry},
  pages     = {185--212},
  publisher = {Springer},
  address   = {New York},
  year      = {2011},
  doi       = {10.1007/978-1-4419-9569-8_10}
}

@article{curtisRobinsonZhou2020,
  author  = {Curtis, Frank E. and Robinson, Daniel P. and Zhou, Baoyu},
  title   = {A self-correcting variable-metric algorithm framework for nonsmooth optimization},
  journal = {IMA Journal Of Numerical Analysis},
  volume  = {40},
  number  = {2},
  pages   = {1154--1187},
  year    = {2020},
  doi     = {10.1093/imanum/drz008}
}

@article{dennisSchnabel1979,
  author  = {Dennis, Jr., J. E. and Schnabel, R. B.},
  title   = {Least change secant updates for quasi-Newton methods},
  journal = {SIAM Review},
  volume  = {21},
  number  = {4},
  pages   = {443--459},
  year    = {1979},
  doi     = {10.1137/1021091}
}

@article{dennisWolkowicz1993,
  author  = {Dennis, Jr., J. E. and Wolkowicz, Henry},
  title   = {Sizing and least-change secant methods},
  journal = {SIAM Journal On Numerical Analysis},
  volume  = {30},
  number  = {5},
  pages   = {1291--1314},
  year    = {1993},
  doi     = {10.1137/0730067}
}

@article{fletcher1970new,
  author  = {Fletcher, Roger},
  title   = {A new approach to variable metric algorithms},
  journal = {The Computer Journal},
  volume  = {13},
  number  = {3},
  pages   = {317--322},
  year    = {1970},
  doi     = {10.1093/comjnl/13.3.317}
}

@article{goldfarb1976factorized,
  author  = {Goldfarb, Donald},
  title   = {Factorized variable metric methods for unconstrained optimization},
  journal = {Mathematics Of Computation},
  volume  = {30},
  number  = {136},
  pages   = {796--811},
  year    = {1976},
  doi     = {10.1090/S0025-5718-1976-0423804-2}
}

@article{lee2014proximal,
  author  = {Lee, Jason D. and Sun, Yuekai and Saunders, Michael A.},
  title   = {Proximal Newton-type methods for minimizing composite functions},
  journal = {SIAM Journal On Optimization},
  volume  = {24},
  number  = {3},
  pages   = {1420--1443},
  year    = {2014},
  doi     = {10.1137/130921428}
}

@article{nesterov2013gradient,
  author  = {Nesterov, Yurii},
  title   = {Gradient methods for minimizing composite functions},
  journal = {Mathematical Programming},
  volume  = {140},
  number  = {1},
  pages   = {125--161},
  year    = {2013},
  doi     = {10.1007/s10107-012-0629-5}
}

@book{nocedalWright2006,
  author    = {Nocedal, Jorge and Wright, Stephen J.},
  title     = {Numerical optimization},
  series    = {Springer Series In Operations Research And Financial Engineering},
  edition   = {2},
  publisher = {Springer},
  address   = {New York},
  year      = {2006},
  doi       = {10.1007/978-0-387-40065-5}
}

@article{orenLuenberger1974,
  author  = {Oren, Shmuel S. and Luenberger, David G.},
  title   = {Self-scaling variable metric ({SSVM}) algorithms: part {I}, criteria and sufficient conditions for scaling a class of algorithms},
  journal = {Management Science},
  volume  = {20},
  number  = {5},
  pages   = {845--862},
  year    = {1974},
  doi     = {10.1287/mnsc.20.5.845}
}

@inproceedings{park2020variable,
  author    = {Park, Youngsuk and Dhar, Sauptik and Boyd, Stephen and Shah, Mohak},
  title     = {Variable metric proximal gradient method with diagonal {Barzilai--Borwein} stepsize},
  booktitle = {ICASSP 2020 -- 2020 IEEE International Conference On Acoustics, Speech And Signal Processing},
  pages     = {3597--3601},
  publisher = {IEEE},
  year      = {2020},
  doi       = {10.1109/ICASSP40776.2020.9054193}
}

@article{powell1978lagrangian,
  author  = {Powell, M. J. D.},
  title   = {Algorithms for nonlinear constraints that use Lagrangian functions},
  journal = {Mathematical Programming},
  volume  = {14},
  number  = {1},
  pages   = {224--248},
  year    = {1978},
  doi     = {10.1007/BF01588967}
}

@article{raydan1997barzilai,
  author  = {Raydan, Marcos},
  title   = {The {Barzilai} and {Borwein} gradient method for the large scale unconstrained minimization problem},
  journal = {SIAM Journal On Optimization},
  volume  = {7},
  number  = {1},
  pages   = {26--33},
  year    = {1997},
  doi     = {10.1137/S1052623494266365}
}

@article{salzo2017variable,
  author  = {Salzo, Saverio},
  title   = {The variable metric forward-backward splitting algorithm under mild differentiability assumptions},
  journal = {SIAM Journal On Optimization},
  volume  = {27},
  number  = {4},
  pages   = {2153--2181},
  year    = {2017},
  doi     = {10.1137/16M1073741}
}

@article{yu2023mini,
  author  = {Yu, Teng-Teng and Liu, Xin-Wei and Dai, Yu-Hong and Sun, Jie},
  title   = {A mini-batch proximal stochastic recursive gradient algorithm with diagonal {Barzilai--Borwein} stepsize},
  journal = {Journal Of The Operations Research Society Of China},
  volume  = {11},
  number  = {2},
  pages   = {277--307},
  year    = {2023},
  doi     = {10.1007/s40305-022-00436-2}
}

@book{grotschel2012geometric,
  title={Geometric algorithms and combinatorial optimization},
  author={Gr{\"o}tschel, Martin and Lov{\'a}sz, L{\'a}szl{\'o} and Schrijver, Alexander},
  volume={2},
  year={2012},
  publisher={Springer Science \& Business Media},
  doi={10.1007/978-3-642-78240-4}
}

@article{blandGoldfarbTodd1981ellipsoid,
author = {Bland, Robert G. and Goldfarb, Donald and Todd, Michael J.},
title = {Feature article—the ellipsoid method: a survey},
journal = {Operations Research},
volume = {29},
number = {6},
pages = {1039-1091},
year = {1981},
doi = {10.1287/opre.29.6.1039}
}
\end{document}